\documentclass[11pt]{article}

\usepackage{amsthm,amsfonts,amssymb,amsmath,oldgerm}
\numberwithin{equation}{section}
\usepackage[thinlines]{easybmat}
\usepackage{caption,graphicx}
\usepackage[text={6in,8in},centering,letterpaper]{geometry}
\usepackage{color}
\usepackage{paralist}

\usepackage{hyperref}

\newcommand{\bq}{\begin{equation}}
\newcommand{\eq}{\end{equation}}
\newcommand{\bqa}{\begin{eqnarray*}}
\newcommand{\eqa}{\end{eqnarray*}}
\newcommand{\Rr}{\mathbb{R}}
\newcommand{\Zz}{\mathbb{Z}}
\newcommand{\Nn}{\mathbb{N}}
\newcommand{\T}{\mathbb{T}}

\newcommand{\la}{\label}

\newcommand{\na}{\nabla}
\newcommand{\be}{\begin{equation}}
\newcommand{\ee}{\end{equation}}
\newcommand{\ba}{\begin{array}{l}}
\newcommand{\ea}{\end{array}}

\theoremstyle{plain}
\newtheorem{theo}{Theorem}[section]
\newtheorem{prop}[theo]{Proposition}
\newtheorem{lemm}[theo]{Lemma}

\theoremstyle{definition}
\newtheorem{rema}[theo]{Remark}

\DeclareMathOperator{\dive}{div}
\DeclareMathOperator{\dist}{dist}

\DeclareMathOperator{\curl}{curl}
\DeclareMathOperator{\supp}{supp}

\DeclareSymbolFont{pletters}{OT1}{cmr}{m}{sl}
\DeclareMathSymbol{s}{\mathalpha}{pletters}{`s}

\def\tt{\theta}

\def\eps{\varepsilon}

\def\na{\nabla}
\def\la{\left\lvert}

\def\le{\leq}

\def\mez{\frac{1}{2}}

\def\ra{\right\rvert}

\def\tdm{\frac{3}{2}}

\def\P{\mathbb P}

\def\cD{\mathcal{D}}
\def\p{\partial}

\def\RR{\mathbb{R}}

\newcommand{\hk}{\hspace*{.15in}}

\begin{document}
\title{On global stability of optimal rearrangement maps }

\author{
Huy Q. Nguyen\footnotemark[1] \and 
Toan T. Nguyen\footnotemark[2]
}

\maketitle

\renewcommand{\thefootnote}{\fnsymbol{footnote}}

\footnotetext[1]{Department of Mathematics, Brown University, Providence, RI 02912. Email: hnguyen@math.brown.edu }

\footnotetext[2]{Department of Mathematics, Penn State University, State College, PA 16803. Email: nguyen@math.psu.edu}

\maketitle

\begin{abstract}

We study the nonlocal vectorial transport equation $\p_ty+ (\P y \cdot \na) y=0$ on bounded domains of $\Rr^d$ where $\P$ denotes the Leray projector. This equation was introduced to obtain the unique optimal rearrangement of  a given map $y_0$ as the infinite time limit of the solution with initial data $y_0$ (\cite{AHT, Macthesis, Brenier09}). We rigorously justify this expectation by proving that for initial maps $y_0$ sufficiently close to  maps with strictly convex potential, the solutions $y$ are global in time and converge exponentially fast to the optimal rearrangement of $y_0$ as time tends to infinity.


\end{abstract}

\section{Introduction} 
 Let $\Omega$ be a bounded domain in $\Rr^d$ equiped with the Lebesgue measure. Two $L^2$ maps $y_1, y_2:\Omega\to \Rr^d$ are rearrangements of each other if they define the same image measure of the Lebesgue measure, i.e.
\[
\int_\Omega f(y_1(x))dx=\int_\Omega f(y_2(x))dx
\]
for all compactly supported continuous function $f:\Rr^d\to \Rr$. A celebrated theorem due to Brenier \cite{Brenier91} asserts that for each $L^2$ map $y_0: \Omega\to \Rr^d$ there exists a unique rearrangement  $y^*$ with {\it convex potential}, i.e. $y^*=\na p^*$ for some convex function $p^*$. Moreover,  among all  possible rearrangements of $y_0$, $y^*$ minimizes the quadratic cost function 
\[
\int_{\Omega}|y(x)-x|^2dx.
\]
  We shall refer to $y^*$ as the {\it optimal rearrangement of $y_0$}. Finding the unique optimal rearrangement $y^*$ for a given map $y_0$ is thus among the main concerns in  optimal transport theory.  As an attempt to get the optimal rearrangement $y^*$ of $y_0$ as an {\it equilibrium state in the infinite time} of a dynamical system that could be efficiently solved by computer,  Angenent, Haker, and Tannenbaum \cite{AHT} (see also McCann \cite{Macthesis, Mac1} and Brenier \cite{Brenier09}) proposed the following nonlocal vectorial  transport model (AHT)
\bq\label{AHT}
\begin{aligned}
\p_ty+u\cdot \nabla y&=0,\\
u&=\P y,
\end{aligned}
\eq
where $y = y(x,t) \in \RR^d$, $x\in \Omega \subset \RR^d$, $t\ge 0$, and  $\P$ denotes the classical Leray projector onto the space of divergence-free vector fields. Throughout the paper, we take either $\Omega = \T^d$ (periodic domain) or a bounded domain in $\RR^d$, $d\ge 2$ with smooth boundary.  The Leray projector $u = \P y$ is defined as follows. For a given map $y:\Omega\to \Rr^d$, we construct the potential $p$ that solves 
\begin{equation}\label{def-P} 
\left\{\begin{aligned}\Delta p &= \nabla \cdot y  \qquad \mbox{in}\quad \Omega 
\\
\frac{\partial p}{\partial n} & = y \cdot n  \qquad \mbox{on}\quad \partial \Omega\end{aligned}
\right.\end{equation}
where $n$ is the unit outward normal to $\p\Omega$. Then we define 
\[
\P y = y - \nabla p.
\]
 As a consequence of the definition, the velocity $u=\P y$ is tangent to the boundary,
\begin{equation}\label{BCs} 
u\cdot n = 0  \qquad \mbox{on} \quad \partial\Omega.\end{equation}  
When $\Omega = \T^d$, the Leray projector $\P =(\P_{ij})$ is simply a Fourier multiplier matrix of order $0$ with 
\[
\P_{ij} (\xi)= 1 - \frac{\xi_i\xi_j}{|\xi|^2},\quad \xi \in \mathbb{Z}^d.
\]


%
Interestingly, the AHT model \eqref{AHT} can also be obtained as the zero inertial limit of generalized (damped) Euler-Boussinesq equations in convection theory  \cite{Brenier09, Bre1}.  In addition, by specifying  $y(x)=(0, \rho(x))$,  \eqref{AHT} reduces to the incompressible porous media (IPM) equations 
\bq\label{IPM}
\begin{cases}
\p_t \rho+(u\cdot\na)\rho =0,\quad x\in \Omega\subset \Rr^2,\\
u+\na p= (0, \rho)^T.
\end{cases}
\eq
Here $\rho$ plays the role of fluid density. Stability of the special solution $\rho_*(x_1, x_2)=x_2$ of \eqref{IPM} has been proved in \cite{Tarek} for $\Omega=\Rr^2,~\T^2$, and in \cite{CCL} for $\Omega=\T \times (-l, l)$ which posses two horizontal boundaries. The presence of boundaries, though only flat boundaries, makes the proof in \cite{CCL} more involved. 

Following \cite{Brenier09} let us explain why \eqref{AHT} is expected to capture the optimal rearrangement of initial maps as steady states in infinite time. First, since the velocity $u$ is divergence-free and tangent to the boundary, we have
\[
\frac{d}{dt}\int_\Omega f(y(x, t))dx=0\quad\forall t>0
\]
for any  compactly supported continuous function $f:\Rr^d\to \Rr$. Integrating this in time we obtain that each $y(t)$, $t>0$ is a rearrangement of $y(0)$. 
Second, it is readily checked that the balance law 
$$ \frac{d}{dt} \int_\Omega \frac12 |y-x|^2 \; dx  = - \int_\Omega |u|^2\; dx$$
holds. In particular, steady states must be gradients since their Leray projections vanish. Conversely, all gradients are clearly steady states of \eqref{AHT}. Now if $y$ is global and the infinite-time limit $y_\infty$ of $y$ exists (in a sufficiently strong topology) then the integral $\int_0^\infty \int_\Omega |u|^2dx dt$ is finite. Consequently, $u$ vanishes as $t\to \infty$ and thus $y_\infty$ must be a gradient, $y_\infty=\na p_\infty$. If we have  in addition that $p_\infty$ is a convex function, then coupling with the fact that $y_\infty$ is a rearrangement of $y(0)$ we conclude by virtue of the aforementioned theorem of Bernier that $y_\infty$ is the unique optimal rearrangement of $y(0)$. The remaining issues in the above argument are global existence and long time behavior for \eqref{AHT}.  On the other hand, the objects that we expect \eqref{AHT} to capture in infinite time are maps with convex potential. A natural problem then is: 
\[
\text{\it Are maps with convex potential globally stable?} 
\]
Our  goal in the present paper is to prove that maps with {\it strictly convex} potential are globally stable. Precisely, our main theorem reads as follows.
\begin{theo}\label{theo:global}
Let $s>1+\frac{d}{2}$ be an integer with $d\ge 2$. Let $\Omega$ be a $C^\infty$ bounded domain in $\Rr^d$.  Consider $y_*=\nabla p_*$ for some strictly convex function $p_*:\Omega\to \Rr$ whose Hessian  satisfies 
\bq\label{convexity}
\na^2 p_*(x)\ge \theta_0\text{Id}\quad\forall x\in \Omega,\quad\tt_0>0.
\eq
Then, there exists a small  positive number $\eps$ depending only on $\theta_0$ and $\Vert y_*\Vert_{H^{s+1}(\Omega)}$ such that for all $y_0\in H^s(\Omega)$ with $\Vert y_0-y_*\Vert_{H^s(\Omega)}\le \eps$,  problem \eqref{AHT}  has a unique global solution $y$. In addition, there is a positive constant $C$ depending only on $\theta_0$ and $\Vert y_*\Vert_{H^{s+1}(\Omega)}$ so that 
\bq\label{stab:est}
\Vert y(t)-y_*\Vert_{H^s(\Omega)}\le C\Vert y_0-y_*\Vert_{H^s(\Omega)}
\eq
and
\bq\label{damping:u}
\Vert\P y(t)\Vert_{H^s(\Omega)}\le C\Vert y_0-y_*\Vert_{H^s(\Omega)}e^{-\frac{\theta_0}{C} t}
\eq
for all $t\ge 0$. Moreover, there exists a strictly convex function $p_\infty:\Omega\to \Rr$ such that
\bq\label{relax}
 \| y(t)-\nabla p_\infty\|_{H^{s-1}(\Omega)}\le Ce^{-\frac{\tt_0t}{C}}\quad\forall t\ge 0.
\eq
In particular,  $\na p_\infty$ is the optimal rearrangement of $y_0$.
\end{theo}
\begin{rema}
  The domain $\Omega$ need not be $C^\infty$ but only $C^{[s]+n_0}$ for a  sufficiently large integer $n_0$.
\end{rema}
For periodic perturbations, our result reads as follows.  
\begin{theo}\label{theo:globalT}
Let $s>1+\frac{d}{2}$ be a real number with $d\ge 2$. Let  $\Omega$ be the box $[0, L]^d$, $L>0$.  Consider $y_*=\nabla p_*$ for some  function $p_*:\Omega\to \Rr$ whose Hessian is $L$-periodic  and satisfies 
\bq\label{convexityT}
\na^2 p_*(x)\ge \theta_0\text{Id}\quad\forall x\in \Omega,\quad\tt_0>0.
\eq
Then, there exists a small  positive number $\eps$ depending only on $\theta_0$ and $\Vert y_*\Vert_{H^{s+1}(\Omega)}$ such that the following holds. For all $y_0\in H^s(\Omega)$ such that $y_0-y_*$ is $L$-periodic and $\Vert y_0-y_*\Vert_{H^s(\Omega)}\le \eps$,  problem \eqref{AHT}  has a unique global solution $y$. In addition, there is a positive constant $C$ depending only on $\theta_0$ and $\Vert y_*\Vert_{H^{s+1}(\Omega)}$ so that $y(t)-y_*$ is $L$-periodic, 
\bq\label{stab:estT}
\Vert y(t)-y_*\Vert_{H^s(\Omega)}\le C\Vert y_0-y_*\Vert_{H^s(\Omega)}
\eq
and
\bq\label{damping:uT}
\Vert\P y(t)\Vert_{H^s(\Omega)}\le C\Vert y_0-y_*\Vert_{H^s(\Omega)}e^{-\frac{\theta_0}{4} t}
\eq
for all $t\ge 0$. Moreover, there exists a strictly convex function $p_\infty:\Omega\to \Rr$ such that
\bq\label{relaxT}
 \| y(t)-\nabla p_\infty\|_{H^{s-1}(\Omega)}\le Ce^{-\frac{\tt_0t}{4}}\quad\forall t\ge 0.
\eq
In particular,  $\na p_\infty$ is the optimal rearrangement of $y_0$.
\end{theo}
\begin{rema}
In Theorem \ref{theo:globalT}, the steady solution $y_*=\na p_*$ need not be periodic, only its Hessian $\na^2p_*$ is required to be periodic and positive. A typical example is $p_*=|x|^2+q_*(x)$ where $q_*$ is $L$-periodic with sufficiently small $\|\na^2q_*\|_{L^\infty(\Omega)}$. The initial map $y_0$ need not be periodic, but if the initial perturbation $y_0-y_*$ is periodic then it remains so for all positive times. 
\end{rema}
The estimates  \eqref{relax} and \eqref{relaxT} exhibit the exponential convergence towards the optimal rearrangement of $y_0$ provided that $y_0$ is sufficiently close to a map with strictly convex potential. This justifies the efficiency of the AHT model \eqref{AHT}. Theorem \ref{theo:global} also provides the first class of time-dependent global solutions to this nonlocal vectorial transport equation for which the issues of global regularity and finite-time blowup remain open. 

Let $y_* = \nabla p_*$ be a steady state of \eqref{AHT} where $p_*$ satisfies the strict convexity condition \eqref{convexity}. Introduce the perturbation $z=y-y_*$. 
Noticing that $\P y_* = 0$, equation \eqref{AHT} yields 
\bq\label{eq:z0}
\begin{aligned}
\p_tz+u\cdot\nabla y_*+u\cdot \nabla z&=0,\\
u&=\P z,
\end{aligned}
\eq
where $u\cdot n =0$ on $\partial \Omega$.  In order to obtain the global stability, some form of decay is needed. Since $z$ is transported, it is not expected to decay.  Our idea is to obtain decay for the divergence-free part $u$ of $z$. Indeed, taking Leray's projection of \eqref{eq:z0} one finds that $u$ obeys
\bq\label{eq:u:intro}
\p_t u+\P (u\cdot\nabla y_*)+\P(u\cdot \nabla z)=0.
\eq
An $L^2$ energy estimate combined with  the strict convexity of $p_*$ and the fact that $\P$ is self-adjoint in $L^2$ shows that $u$ decays exponentially when measured in $L^2$. We need however decay of high Sobolev norms of $u$ in order to close the nonlinear iteration. 
In performing a direct $H^s$ energy estimate for $u$ at the level of \eqref{eq:u:intro}, there are at least two difficulties:
\begin{itemize}
\item[(i)] the term $u\cdot \nabla z$ would induce a loss of derivatives due to the presence of $\na z$;
\item[(ii)] to reveal the damping mechanism due to $\na y_*=\na^2p_*\ge \theta_0\text{Id}$, one needs to make appear the term $D^s u\cdot \na y_*$ where $D^s$ denotes any partial derivatives of order $s$. However, in the presence of boundaries,  $D^s$ do not commute with $\P$. Moreover, in general the commutator  $[D^s, \P]$ does not exhibit a gain of derivative, and hence is of the same order as the damping term.
\end{itemize}
\hk To handle (i) we commute $\P$ with $u\cdot   \na$ as follows
\bq\label{eq:u0}
\p_tu+\P(u\cdot \nabla y_*)+u\cdot\nabla u+[\P, u\cdot\nabla] z=0.
\eq
 The new nonlinear term $u\cdot \na u$ is now an advection term, and thus does not induce any loss of derivatives. However, a gain of one derivative in $[\P, u\cdot\nabla] z$ is then needed. As mentioned in (ii), such a gain is not true in general for $[\P, \p_j]$. Interestingly, if one replaces partial derivatives $\p_j$ with $u\cdot \na$, this holds even in domains with boundary, provided only that $u$ is tangent to the boundary. This is the content of the next theorem, which is of independent interest. Throughout this paper we denote
\[
  \T^d:=(\Rr/L\Zz)^d.
  \]
\begin{theo}\label{theo:cmt} Let $s> 1+\frac{d}{2}$ be an integer with $d\ge 2$. Consider  $\Omega = \T^d$ or $\Omega$ a bounded domain in $\RR^d$ with smooth boundary. Let $\P$ denote the Leray projector associated to $\Omega$. Then, for any vector fields $u$, $z\in H^s(\Omega; \Rr^d)$ with $u\cdot n\vert_{\p\Omega}=0$ when $\partial\Omega \not =\emptyset$, the commutator estimate
\bq\label{cmt:elliptic1}
\| [\P, u\cdot\nabla]z\|_{H^s(\Omega)}\le C\|u\|_{H^s(\Omega)}\| z\|_{H^s(\Omega)}
\eq
holds for some universal constant $C$. 
\end{theo}
Regarding the difficulty (ii), we observe that ``tangential derivatives'' commute nicely with the Leray projector while ``normal derivatives'' do not. We then introduce a boundary adapted system of derivatives $\cD^s$ (see Section \ref{derivativesystem}) which are defined everywhere and become the usual  tangential and normal derivatives when restricted to the boundary. Next, to avoid the commutator $[\cD^s, \P]$ when dealing with the nonlocal term $\P (u\cdot\na y_*)$ we write
\[
\int_\Omega \cD^su\cdot \cD^s\P (u\cdot \na y_*)dx=\int_\Omega \cD^su\cdot \cD^s (u\cdot \na y_*)dx+\int_\Omega \cD^su\cdot \cD^s(\P-\text{Id})(u\cdot \na y_*)dx
\] 
and notice a special structure in the second integral.  This allows us  to prove a {\it hierarchy of estimates} for the velocity $u$, ordered by {\it the number of normal derivatives} in $\cD^s$, and hence to close our nonlinear iteration.

For the proof of Theorem \ref{theo:global} we will need the local well-posedness of the AHT model \eqref{AHT} in Sobolev spaces:
\begin{theo}\label{theo:lwp}
Let $\Omega$ be a bounded domain in $\Rr^d$, $d\ge 2$ with smooth boundary or periodic boundary conditions. Let $s>1+\frac{d}{2}$ be an integer. Then for any initial data $z_0\in H^s(\Omega)$, there exist a positive time $T$ depending only on $\| z_0\|_{H^s(\Omega)}$ and a unique solution $z\in C([0, T]; H^s(\Omega))$ of \eqref{AHT}.
\end{theo}
Local well-posedness of \eqref{AHT} in H\"older spaces $C^{1, \alpha}(\Omega)$ has been obtained in \cite{AHT}. Since the velocity $u$ has the same Sobolev regularity as the unknown $z$,  the proof of Theorem 1.6 is standard via energy methods, and thus will be skipped.

The paper is organized as follows. Section \ref{sec-commutator} is devoted to various commutator estimates involving the Leray projector. Theorem \ref{theo:globalT} is proved in Section \ref{sec-proofT}, and Theorem \ref{theo:global} is proved  in Section \ref{sec-proofB}.

Throughout this paper, we denote  by $\p_j$,  $j\in\{1,...,d\}$ the $j$th partial derivative and by $D^m$ any partial derivatives of order $m\in \Nn$.

\section{Commutator estimates}\label{sec-commutator}
\subsection{A boundary adapted system of derivatives}\label{derivativesystem}
For simplicity, we assume from now on that $\Omega$ is  a $C^\infty$ domain. Let $\delta(x)=\dist(x, \p\Omega)$ be the distance function. There exists a small number $\kappa>0$ such that $\delta$ is $C^\infty$  in the neighborhood of the  region
\[
\Omega_{3\kappa}=\{ x\in \Omega: \delta(x)\le 3\kappa\}
\]
and $\na \delta(x) \ne 0$ for any $x\in \Omega_{3\kappa}$. Note that the unit  outward normal $n(x)=-\na\delta(x)$ for $x\in \p\Omega$. We thus can extend $n$ to $\Omega_{3\kappa}$ by setting
\[
n(x)=-\frac{\na\delta(x)}{|\na \delta(x)|},\quad x\in \Omega_{3\kappa}.
\]
For each $x\in \Omega_{3\kappa}$, we can choose $\tau(x)=\{\tau_j(x): j=1,...,d-1\}$ an  orthonormal basis of $(n(x))^\perp$ in $\Rr^d$ such that $\tau_j\in C^\infty(\Omega_{3\kappa})$.  

Next we fix a cutoff function $\chi_1:\overline\Omega\to [0, 1]$  satisfying 
\bq\label{def:chi1}
\chi_1\equiv 1\quad\text{in a neighborhood of}~\Omega_{2\kappa},\qquad  \chi_1\equiv 0\quad\text{in}~\overline\Omega\setminus \Omega_{3\kappa}.
\eq
 For a vector field $v:\Omega\to \Rr^d$ we define  its weighted normal and tangential components respectively by 
\[
v_n(x)=\chi_1(x)v(x)\cdot n(x),\quad v_{\tau_j}(x)=\chi_1(x)v(x)\cdot \tau_j(x),\quad j=1,...,d-1
\]
for $x\in \Omega$. In particular, $v=v_n n+\sum_{j=1}^{d-1}v_{\tau_j}\tau_j$ in $\Omega_{2\kappa}$. In the special case of gradient vectors $\na f$ where $f:\Omega\to \Rr$, we write
\[
\p_n f=(\na f)_n,\quad \p_{\tau_j}f=(\na f)_{\tau_j},\quad j=1,...,d-1.
\]
Both $\p_n f$ and $\p_\tau f$ are defined over  $\Omega$ and become the usual normal and tangential derivatives  when restricted to the boundary. Note in addition that
\bq\label{decompose:na1}
\na f=n\p_n f+\sum_{j=1}^{d-1}\tau_j\p_{\tau_j}f\quad\text{in} ~\Omega_{2\kappa}.
\eq
For a vector field $v:\Omega\to \Rr^d$ we write $\p_n v=(\na v) \cdot n$ and similarly for $\p_{\tau_j}v$. Then we have
\bq\label{nav:tn}
|\na v|^2=\sum_{i=1}^d|\na v_i|^2=\sum_{i=1}^d(|\p_n v_i|^2+\sum_{j=1}^{d-1}|\p_{\tau_j} v_i|^2)=|\p_n v|^2+\sum_{j=1}^{d-1}|\p_{\tau_j} v|^2
\eq
for $x\in\Omega_{2\kappa}$.
\begin{lemm}
For $v:\Omega\to \Rr^d$ and $f:\Omega\to \Rr$ we have
\bq\label{pnt}
\p_nv\cdot n+\sum_{j=1}^{d-1}\p_{\tau_j}v\cdot\tau_j=\dive v
\eq
and
\bq\label{pnt2}
\p_n^2f+\sum_{j=1}^{d-1}\p^2_{\tau_j}f=\Delta f+\na f\cdot (n\cdot \na) n+\sum_{j=1}^{d-1}\na f\cdot (\tau_j\cdot \na) \tau_j
\eq
at any $x\in \Omega_{2\kappa}$.
\end{lemm}
\begin{proof}
We first notice that since $\chi_1\equiv 1$  in $ \Omega_{2\kappa}$. If $R$ denotes the matrix whose first $d-1$ columns are $\tau_1,...,\tau_{d-1}$ and whose $d$th column is $n$, then $R$ is orthonormal; that is, $RR^T=\text{Id}$. Using this and the above definitions of $\p_n$ and $\p_{\tau_j}$ we have
\[
\begin{aligned}
\p_nv\cdot n+\sum_{j=1}^{d-1}\p_{\tau_j}v\cdot\tau_j&=\sum_{k, \ell=1}^d\p_kv_\ell n_kn_\ell+\sum_{j=1}^{d-1}\sum_{k, \ell=1}^d\p_kv_\ell \tau_{j,k}\tau_{j,\ell}\\
&=\sum_{k, \ell=1}^d \p_ku_\ell \sum_{j=1}^dR_{k, j}R^T_{j, \ell}\\
&=\sum_{k, \ell=1}^d \p_ku_\ell \delta_{k, \ell}\\
&=\dive u.
\end{aligned}
\]
Similarly, we have
\[
\begin{aligned}
\p_n^2f+\sum_{j=1}^{d-1}\p^2_{\tau_j}f&=\sum_{k, \ell=1}^d\p_k\p_\ell fn_kn_\ell+\sum_{j=1}^{d-1}\sum_{k, \ell=1}^d\p_k\p_\ell f\tau_{j,k}\tau_{j,\ell}+\na f\cdot (n\cdot \na) n+\sum_{j=1}^{d-1}\na f\cdot (\tau_j\cdot \na) \tau_j\\
&=\sum_{k, \ell=1}^d \p_k\p_\ell f\delta_{k, \ell}+\na f\cdot (n\cdot \na) n+\sum_{j=1}^{d-1}\na f\cdot (\tau_j\cdot \na) \tau_j\\
&=\Delta f+\na f\cdot (n\cdot \na) n+\sum_{j=1}^{d-1}\na f\cdot (\tau_j\cdot \na) \tau_j.
\end{aligned}
\]
\end{proof}
\subsection{Proof of Theorem \ref{theo:cmt}}
In this section, we study the commutator term $[\P, u\cdot\nabla] z$ appearing in \eqref{eq:u0} and prove Theorem \ref{theo:cmt}. We start with the periodic case $\Omega = \T^d$, $d\ge 2$.

\begin{lemm}\label{lemm:comm}
Let $s>1+\frac{d}{2}$ be a real number and let $\P$ be the Leray projector. Then, for any $u,z\in H^s(\T^d)$, there holds
\[
\Vert [\P, u\cdot\nabla ] z\Vert_{H^s(\T^d)}\le C\Vert u\Vert_{H^s}\Vert z\Vert_{H^s(\T^d)}
\]
for some universal constant $C$. 
\end{lemm}
\begin{proof}
For any two functions $f,~g: \Rr^d\to \Rr$, we have the following Bony's decomposition (see \cite{Bony, BCD})
\[
fg=T_fg+T_gf+R(f, g)
\]
where $T_fg$ and $T_gf$ are paraproducts, so that the following hold: 
\begin{itemize}
\item Pararoduct estimates:
\bq\label{paraprod}
\Vert T_fg\Vert_{H^r}\le C\Vert f\Vert_{L^\infty}\Vert g\Vert_{H^r}\quad \forall r\in \Rr.
\eq
\item Reminder estimates:
\bq\label{Bonyr}
\Vert R(f, g)\Vert_{H^{r_1+r_2-\frac{d}{2}}}\le C\Vert  f\Vert_{H^{r_1}}\Vert g\Vert_{H^{r_2}}
\eq
for all $r_1$ and $r_2$ in $\Rr$ satisfying $r_1+r_2>0$.
\item Commutator estimates:
\bq\label{paracomm}
\Vert [T_f, m(D)]g\Vert_{H^{r-k+\alpha}}\le C\Vert f\Vert_{W^{\alpha, \infty}}\Vert g\Vert_{H^r}
\eq
for any homogeneous operator $m(D)$ of order $k$, $\alpha\in (0, 1]$, and $r\in \Rr$.
\end{itemize}
Recall that on $\T^d$, the Leray projector $\P =(\P_{ij})$ is a Fourier multiplier matrix of order $0$. In addition, we note that $[\P,\nabla] =0$. Denoting by $v_k$ the kth component of the vector $v$, we have
\begin{align*}
([\P, u\cdot \nabla ]z)_k&=(\P(u\cdot\nabla z))_k-(u\cdot \nabla\P z)_k\\
&=\P_{kj}(u_i\p_iz_j)-u_i\p_i\P_{kj}z_j\\
&=\P_{kj}(u_i\p_iz_j)-u_i\P_{kj}\p_iz_j\\
&=\P_{kj}(T_{u_i}\p_iz_j)+\P_{kj}(T_{\p_iz_j}u_i)+\P_{kj} R(u_i, \p_iz_j) \\&\quad -T_{u_i}\P_{kj}\p_iz_j-T_{\P_{kj}\p_iz_j}u_i-R(u_i, \P_{kj}\p_iz_j)\\
&=[\P_{kj}, T_{u_i}]\p_iz_j+\P_{kj}(T_{\p_iz_j}u_i)+\P_{kj} R(u_i, \p_iz_j)-T_{\P_{kj}\p_iz_j}u_i-R(u_i, \P_{kj}\p_iz_j).
\end{align*}
Here and in what follows, the summation in $i,j$ is understood. Since $\P$ is a Fourier multiplier of oder $0$ and $T_{u_i}$ has order $0$, the commutator estimate \eqref{paracomm} yields
\[
\Vert [\P_{kj}, T_{u_i}]\p_iz_j\Vert_{H^s}\le C\Vert u\Vert_{W^{1, \infty}}\Vert \nabla z\Vert_{H^{s-1}}\le C\Vert u\Vert_{H^s}\Vert z\Vert_{H^s}
\]
where we used the embedding $H^s\subset W^{1, \infty}$ for $s>1+\frac{d}{2}$. The paraproduct terms $\P_{kj}(T_{\p_iz_j}u_i)$ and $T_{\P_{kj}\p_iz_j}u_i$ are estimated in $H^s$ by means of the paraproduct rule \eqref{paraprod}, the embedding $H^s\subset W^{1, \infty}$, and the fact that $\P$ is continuous from $H^s$ to $H^s$. Finally, the reminder terms can be treated using \eqref{Bonyr} as follows
\[
\Vert R(u_i, \p_iz_j)\Vert_{H^s}\le \Vert R(u_i, \p_iz_j)\Vert_{H^{s+s-1-\frac{d}{2}}}\le C\Vert u\Vert_{H^s}\Vert \nabla z\Vert_{H^{s-1}}\le C\Vert u\Vert_{H^s}\Vert z\Vert_{H^s},
\]
which gives the lemma. 
\end{proof}

We next turn to the case when $\Omega$ has a boundary. Fix an integer $s>1+\frac{d}{2}$. By definition, we write $z=\P z+\nabla f$, where $f$ solves
\bq\label{cmt:elliptic11}
\begin{cases}
\Delta f= \dive z\quad\text{in}~\Omega,\\
\frac{\p f}{\p n}=z\cdot n\quad\text{on}~\p\Omega.
\end{cases}
\eq
In particular, the standard elliptic regularity theory yields $\| f\|_{H^{s+1}(\Omega)} \le C \| z\|_{H^s(\Omega)}$. Similarly, we write $(u\cdot \nabla)z=\P((u\cdot \nabla)z)+\nabla g$, where $g$ solves
\bq\label{cmt:elliptic2}
\begin{cases}
\Delta g= \dive((u\cdot \nabla)z) \quad\text{in}~\Omega,\\
\frac{\p g}{\p n}=(u\cdot \nabla)z\cdot n\quad\text{on}~\p\Omega.
\end{cases}
\eq
Combining, we have
\begin{equation}\label{comm}\begin{aligned}
~[\P, u\cdot\nabla]z&=\P((u\cdot \nabla) z)-(u\cdot\nabla)\P z\\
&=(u\cdot \nabla)z-\nabla g-(u\cdot \nabla)(z-\nabla f)\\
&=(u\cdot \nabla)(\nabla f)-\nabla g.
\end{aligned}
\end{equation}
We shall bound the $H^s$ norm of $[\P, u\cdot\nabla]z$, using the following elliptic estimate
\bq\label{bd-h}
\begin{aligned}
\| h\|_{H^s(\Omega)}&\le C\| \dive h\|_{H^{s-1}(\Omega)}+C\| \curl h\|_{H^{s-1}(\Omega)}+C\| h\cdot n\|_{H^{s-\mez}(\p\Omega)},\end{aligned}
\eq
for $h = [\P, u\cdot\nabla]z$, where the terms on the right hand side are estimated in the following  lemmas.

\begin{lemm}
 There exists a positive constant $C$ such that
\bq\label{est:divh}
\| \dive ([\P, u\cdot\nabla]z)\|_{H^{s-1}(\Omega)}\le C\| u\|_{H^s(\Omega)} \|z\|_{H^{s}(\Omega)}
\eq
and 
\bq\label{est:curlh}
\| \curl  ([\P, u\cdot\nabla]z)\|_{H^{s-1}(\Omega)}\le C\| u\|_{H^s(\Omega)}\|z\|_{H^{s}(\Omega)}.
\eq
\end{lemm}
\begin{proof}
In view of \eqref{comm}, we compute, using \eqref{cmt:elliptic11}, 
\begin{align*}
\dive ((u\cdot \nabla)(\nabla f))
=\na u: (\na \otimes \na)f+u\cdot \nabla \dive z.
\end{align*}
On the other hand, using equation \eqref{cmt:elliptic2}, we have 
\begin{align*}
\dive (\nabla g)&=\Delta g
=\dive((u\cdot \nabla)z).
=\na u:(\na z)^T.
\end{align*}
Combining, we have 
\bq\label{divh}
\dive ([\P, u\cdot\nabla]z)= \na u :[(\na \otimes \na)f-(\na z)^T].
\eq
The estimate \eqref{est:divh} thus follows directly from \eqref{divh}, upon using the fact that $H^{s-1}(\Omega)$ is an algebra and the elliptic estimates $\| f\|_{H^{s+1}(\Omega)} \le C \| z\|_{H^s(\Omega)}$.

Next, in view of \eqref{comm}, we write 
\[
 [\P, u\cdot\nabla]z=\nabla(u\cdot \nabla f-g)-\nabla u_k \partial_k f
\]
which gives $\curl  ([\P, u\cdot\nabla]z) =\nabla u_k\times \partial_k \nabla f$. The estimate \eqref{est:curlh} then follows from elliptic estimates as before.  
\end{proof}
\begin{lemm}
There exists a positive constant $C$ such that
\bq\label{est:hn}
\|  [\P, u\cdot\nabla]z \cdot n\|_{H^{s-\mez}(\p\Omega)}\le C\| u\|_{H^s(\Omega)} \|z\|_{H^{s}(\Omega)}.
\eq
\end{lemm}
\begin{proof} We use the decomposition $u = u_n n + \sum_{j=1}^{d-1}u_{\tau_j} \tau_j$  in $\Omega_{2\kappa}$. Then we compute 
\begin{align*}
(u\cdot \nabla)z\cdot n
&=(u\cdot \nabla)(z\cdot n)-(u\otimes z):\nabla n\\
&=u_n (n\cdot \nabla)(z\cdot n)+\sum_{j=1}^{d-1}u_{\tau_j}(\tau_j\cdot \nabla)(z\cdot n)-(u\otimes z):\nabla n
\end{align*}
in $\Omega_{2\kappa}$. Since $u\cdot n=0$ on $\p\Omega$, we have $u_n=0$ on $\p\Omega$. Taking the trace of the above equation on $\p\Omega$ and recalling \eqref{cmt:elliptic2}, 
we get 
$$
\nabla g\cdot n = (u\cdot \nabla)z\cdot n=\sum_{j=1}^{d-1}u_{\tau_j}\p_{\tau_j} (z\cdot n)-(u\otimes z):\nabla n\quad\text{on}~\p\Omega.
$$
Similarly, on $\partial \Omega$, we have \begin{align*}
(u\cdot \nabla)(\nabla f)\cdot n
&=\sum_{j=1}^{d-1}u_{\tau_j} \p_{\tau_j} (\nabla f\cdot n)-(u\otimes \nabla f):\nabla n.
\end{align*}
Recalling \eqref{comm} and using the boundary condition \eqref{cmt:elliptic11}, which gives $\partial_{\tau_j} (\nabla f \cdot n) = \partial_{\tau_j} (z\cdot n)$ on $\p\Omega$, we obtain 
$$\begin{aligned}
~[\P, u\cdot\nabla]z\cdot n &=(u\cdot \nabla)(\nabla f) \cdot n-\nabla g\cdot n = [u\otimes (z-\nabla f)]:\nabla n 
\end{aligned}
$$
on $\partial \Omega$. Using the trace inequality, we bound
\begin{align*}
\| [\P, u\cdot\nabla]z\cdot n\|_{H^{s-\mez}(\p\Omega)}&\le C\|u\otimes (z-\nabla f)\|_{H^s(\Omega)}\\
&\le  C\|u\|_{H^s(\Omega)}\big(\| z\|_{H^s(\Omega)}+\|f\|_{H^{s+1}(\Omega)}\big)
\end{align*}
which gives \eqref{est:hn}, upon recalling the elliptic estimates $\| f\|_{H^{s+1}} \le C \| z\|_{H^s}$.
\end{proof}
\subsection{Commutators between the Leray projector and tangential derivatives}\label{sec:tn}

\begin{prop}\label{prop:estuO}
Let  $m\ge 2$ be an integer.  There exists a constant $C>0$ such that 
\[
\| [P, \P] u\|_{L^2(\Omega)}\le C\| u\|_{H^{m-1}(\Omega)}
\]
for any $P\in \{\Pi_{j=1}^m\p_{\sigma_j}: \sigma_j\in \{\tau_1,...,\tau_{d-1}\}\}$ and any vector field $u\in H^{m-1}(\Omega)$.
\end{prop}
\begin{proof}
Without loss of generality we consider $P=\p_{\tau_1}^m$. In view of the identity 
\[
[\p_{\tau_1}^{q+1}, \P] u=[\p_{\tau_1}^q, \P] \p_{\tau_1}u+\p_{\tau_1}^q[\p_{\tau_1}, \P]u, \quad q\ge 1
\]
and by induction in $m$, it suffices to prove that
\bq\label{cmt:t:00}
\| [\p_{\tau_1}, \P] \p_{\tau_1}u\|_{L^2(\Omega)}\le C\| u\|_{H^1(\Omega)}
\eq
and
\bq\label{cmt:t:0}
\| [\p_{\tau_1}, \P] u\|_{H^j(\Omega)}\le C\| u\|_{H^{j}(\Omega)}\quad\forall j \ge 1.
\eq
 To this end, for any vector field $v$, we write $\P v=v-\na f$ and $\P (\p_{\tau_1}v)=\p_{\tau_1}v-\na g$ where $f$ and $g$ solve 
\[
\begin{cases}
\Delta f= \dive v\quad\text{in}~\Omega,\\
\frac{\p f}{\p n}=v\cdot n\quad\text{on}~\p\Omega
\end{cases}
\]
and 
\[
\begin{cases}
\Delta g= \dive (\p_{\tau_1}v)\quad\text{in}~\Omega,\\
\frac{\p g}{\p n}=(\p_{\tau_1}v)\cdot n\quad\text{on}~\p\Omega
\end{cases}
\]
respectively. Then 
\bq\label{cmt:t:form}
[\p_{\tau_1}, \P] v=\na g-\p_{\tau_1}\na f=\na(g-\p_{\tau_1}f)-[\p_{\tau_1}, \na]f.
\eq
  We compute
\[
\begin{aligned}
&\Delta \p_{\tau_1}f=\p_{\tau_1}\Delta f+\Delta (\chi_1\tau_1)\cdot \na f+2\na (\chi_1\tau_1): \na\na f,\\
&\dive(\p_{\tau_1}v)=\p_{\tau_1}\dive v+\na v: (\na (\chi_1\tau_1))^T,
\end{aligned}
\]
where $\chi_1$ is defined as in \eqref{def:chi1}. As a consequence, $h:=g-\p_{\tau_1}f$ satisfies
\bq\label{cmt:t:eq}
\Delta h=\na v: (\na (\chi_1\tau_1))^T-\Delta (\chi_1\tau_1)\cdot \na f-2\na (\chi_1\tau_1): \na\na f \quad\text{in}~\Omega.
\eq
Regarding the boundary condition, we have
\[
\begin{aligned}
&\p_{\tau_1}(v\cdot n)=\p_{\tau_1}v\cdot n+\na n: (v\otimes \tau_1),\\
&\p_n(\p_{\tau_1}f)=\na \na f: (n\otimes \tau_1)+\na \tau_1:(\na f\otimes n),\\
&\p_{\tau_1}(\p_nf)=\na \na f: (n\otimes \tau_1)+\na n:(\na f\otimes \tau_1)
\end{aligned}
\]
in $\Omega$. This yields
\bq\label{cmt:t:bc}
\begin{aligned}
\p_nh&=\p_ng-\p_{\tau_1}\p_nf-\na \tau_1:(\na f\otimes n)+\na n:(\na f\otimes \tau_1)\\
&=(\p_{\tau_1}v)\cdot n-\p_{\tau_1}(v\cdot n)-\na \tau_1:(\na f\otimes n)+\na n:(\na f\otimes \tau_1)\\
&=-\na n: (v\otimes \tau_1)-\na \tau_1:(\na f\otimes n)+\na n:(\na f\otimes \tau_1) \quad\text{on}~\p\Omega.
\end{aligned}
\eq
In addition, elliptic estimates combined with trace inequalities 
\[
\| v\|_{H^{\ell-\tdm}(\p\Omega)}\le C\| v\|_{H^{\ell-1}(\Omega)}\quad\forall \ell\ge 2
\]
yield
\bq\label{cmt:ef}
\| f\|_{H^\ell(\Omega)}\le C\| v\|_{H^{\ell-1}(\Omega)}\quad\forall \ell\ge 2.
\eq
{\it Proof of \eqref{cmt:t:0}.} In view of \eqref{cmt:t:eq}, \eqref{cmt:t:bc} we deduce using elliptic estimates, trace inequalities and \eqref{cmt:ef} that for any $\ell\ge 2$, 
\bq\label{cmt:t:esth}
\begin{aligned}
\| h\|_{H^\ell(\Omega)}&\le C\| \na v: (\na (\chi_1\tau_1))^T-\Delta (\chi_1\tau_1)\cdot \na f-2\na (\chi_1\tau_1): \na\na f\|_{H^{\ell-2}(\Omega)}\\
&\quad +C\|-\na n: (v\otimes \tau_1)-\na \tau_1:(\na f\otimes n)+\na n:(\na f\otimes \tau_1)\|_{H^{\ell-\tdm}(\p\Omega)}\\
&\le C\| \na v: (\na (\chi_1\tau_1))^T-\Delta (\chi_1\tau_1)\cdot \na f-2\na (\chi_1\tau_1): \na\na f\|_{H^{\ell-2}(\Omega)}\\
&\quad +C\|-\na n: (v\otimes \tau_1)-\na \tau_1:(\na f\otimes n)+\na n:(\na f\otimes \tau_1)\|_{H^{\ell-1}(\Omega)}\\
&\le C\| v\|_{H^{\ell-1}(\Omega)}+C\|  f\|_{H^{\ell}(\Omega)}\\
&\le C'\| v\|_{H^{\ell-1}(\Omega)}.
\end{aligned}
\eq
Note that the trace inequality used in the second inequality in \eqref{cmt:t:esth} does not hold when $\ell=1$. Now for any $j\ge 1$ using \eqref{cmt:t:form}, \eqref{cmt:ef} and \eqref{cmt:t:esth} with $\ell=j+1\ge 2$ together with the estimate
\[
\| [\p_{\tau_1}, \na]f\|_{H^j(\Omega)}\le C\| f\|_{H^{j+1}(\Omega)}
\]
we obtain 
\[
\begin{aligned}
\| [\p_{\tau_1}, \P]v\|_{H^j(\Omega)}&\le \|  h\|_{H^{j+1}(\Omega)}+\| [\p_{\tau_1}, \na]f\|_{H^j(\Omega)}\\
&\le C\| v\|_{H^j(\Omega)}+C\| f\|_{H^{j+1}(\Omega)}\\
&\le C'\| v\|_{H^j(\Omega)}
\end{aligned}
\]
which is the desired estimate \eqref{cmt:t:0} if we set $v=u$. 
\vskip 0.5cm
{\it Proof of  \eqref{cmt:t:00}.} Again, we use the equations \eqref{cmt:t:eq}, \eqref{cmt:t:bc} with $v=\p_{\tau_1}u$ and $H^1$ elliptic estimate for the Neumann problem to have
\[
\begin{aligned}
\| h\|_{H^1(\Omega)}&\le C\| \na v: (\na (\chi_1\tau_1))^T-\Delta (\chi_1\tau_1)\cdot \na f-2\na (\chi_1\tau_1): \na\na f\|_{H^{-1}(\Omega)}\\
&\quad +C\|-\na n: (v\otimes \tau_1)-\na \tau_1:(\na f\otimes n)+\na n:(\na f\otimes \tau_1)\|_{H^{-\mez}(\p\Omega)}\\
&\le C'\| v\|_{L^2(\Omega)}+C'\| f\|_{H^1(\Omega)}+C'\| v\|_{H^{-\mez}(\p\Omega)}+C'\| \na f\|_{H^{-\mez}(\p\Omega)}.
\end{aligned}
\]
Since $v=\p_{\tau_1}u$ we have $\| v\|_{L^2(\Omega)}\le C\| u\|_{H^1(\Omega)}$ and
\[
\| v\|_{H^{-\mez}(\p\Omega)}\le C\| u\|_{H^{\mez}(\p\Omega)}\le C'\| u\|_{H^1(\Omega)}.
\]
Moreover, using \eqref{decompose:na1} and the Neumann boundary condition for $f$ we can write
\[
\na f=\sum_{j=1}^{d-1}\tau_j\p_{\tau_j}f +n\p_nf =\sum_{j=1}^{d-1}\tau_j\p_{\tau_j}f+n(\p_{\tau_1}u \cdot n)\quad\text{on}~\p\Omega.
\]
This implies 
\[
\begin{aligned}
\| \na f\|_{H^{-\mez}(\p\Omega)}&\le C\sum_{j=1}^{d-1}\| \p_{\tau_j}f\|_{H^{-\mez}(\p\Omega)}+C\| \p_{\tau_1}u\|_{H^{-\mez}(\p\Omega)}\\
&\le C'\| f\|_{H^{\mez}(\p\Omega)}+C\| u\|_{H^{\mez}(\p\Omega)}\\
&\le C''\| f\|_{H^1(\Omega)}+C\| u\|_{H^1(\p\Omega)}.
\end{aligned}
\]
Thus, we obtain
\[
\| h\|_{H^1(\Omega)}\le C\| u\|_{H^1(\Omega)}+C\| f\|_{H^1(\Omega)}.
\]
The $H^1$ elliptic estimate for $f$ gives
\[
\begin{aligned}
\| f\|_{H^1(\Omega)}&\le C\| \p_{\tau_1}u\|_{L^2(\Omega)}+C\| \p_{\tau_1}u\cdot n\|_{H^{-\mez}(\p\Omega)}\\
 &\le C\| \p_{\tau_1}u\|_{L^2(\Omega)}+C'\| u\|_{H^{\mez}(\p\Omega)}\\
 &\le C''\| u\|_{H^1(\Omega)}.
\end{aligned}
\]
Consequently
\[
\| h\|_{H^1(\Omega)}\le C\| u\|_{H^1(\Omega)}
\]
which combined with the commutator estimate 
\[
\| [\p_{\tau_1}, \na]f\|_{L^2(\Omega)}\le C\| f\|_{L^2(\Omega)}\le C\| f\|_{H^1(\Omega)}
\]
  completes the proof of \eqref{cmt:t:00}.
\end{proof}
Next we fix a cutoff function $\chi_2:\overline\Omega\to [0, 1]$  satisfying 
\bq\label{def:chi2}
\chi_2\equiv 0\quad\text{in}~ \Omega_{\kappa},\qquad \chi_2\equiv 1\quad\text{in}~ \overline\Omega\setminus \Omega_{2\kappa}.
\eq
\begin{prop}\label{cmt:Pna}
Let $m\ge 1$ be an integer. There exists a constant $C>0$ such that
\[
\|[\chi_2D^m, \P]u\|_{L^2(\Omega)}\le C\| u\|_{H^{m-1}(\Omega)}
\]
for any vector field $u\in H^{m-1}(\Omega)$. 
\end{prop}
\begin{proof}
Without loss of generality we consider $D^m=\p_1^m$. We have 
\[
\begin{aligned}
[\chi_2\p_1^m, \P]u&=\chi_2\p_1^m(u-\na f)-[\chi_2\p_1^mu-\na g]\\
&=\chi_2\p_1^m\na f-\na g\\
&=\na(\chi_2\p_1^mf-g)-\na \chi_2\p_1^mf.
\end{aligned}
\]
 where $f$ and $g$ solve
 \[
\begin{cases}
\Delta f= \dive u\quad\text{in}~\Omega,\\
\frac{\p f}{\p n}=u\cdot n\quad\text{on}~\p\Omega
\end{cases}
\]
and 
\[
\begin{cases}
\Delta g= \dive (\chi_2\p_{\tau_1}^mu)\quad\text{in}~\Omega,\\
\frac{\p g}{\p n}=(\chi_2\p_{\tau_1^m}u)\cdot n\quad\text{on}~\p\Omega
\end{cases}
\]
respectively. By elliptic estimates for $f$ we have
\bq\label{cmt:Pna1}
\| \na \chi_2\p_1^mf\|_{L^2(\Omega)}\le C\| u\|_{H^{m-1}(\Omega)}.
\eq
Setting $h=\chi_2\p_1^m\na f-g$ we compute
\[
\begin{aligned}
\Delta h&=\Delta\chi_2\p_1^mf+2\na\chi_2\cdot \na\p_1^mf-\na\chi_2\cdot \p_1^mu\\
&=\Delta\chi_2\p_1^mf+2\dive(\na\chi_2 \p_1^mf)-2\Delta \chi_2\p_1^mf-\p_j(\na\chi_2\cdot \p_1^{m-1}u)+\na \p_1\chi_2\cdot \p_1^{m-1}u.
\end{aligned}
\]
On the other hand, since $\chi_2\equiv 0$ near $\p\Omega$, $h\equiv 0$ near $\p\Omega$. Thus, standard elliptic estimates give
\bq\label{cmt:Pna2}
\| h\|_{H^1(\Omega)}\le C\| u\|_{H^{m-1}(\Omega)}.
\eq
A combination of \eqref{cmt:Pna1} and \eqref{cmt:Pna2} concludes the proof.
\end{proof}
\section{Proof of Theorem \ref{theo:globalT}}\label{sec-proofT}


Let us start with a priori estimates for the perturbation $z = y - y_*$, which solves 
\bq\label{eq:z}
\begin{aligned}
\p_tz+u\cdot\nabla y_*+u\cdot \nabla z&=0,\qquad
u=\P z
\end{aligned}
\eq
in $\Omega=[0, L]^d$. We consider $p_*$ such that $\na y_*=\na^2p_*$ is $L$-periodic. In what follows, $u$ and $z$ are  $L$-periodic smooth solutions to \eqref{eq:z}, and $s$ is real number satisfying $s>1 + \frac{d}{2}$. We shall denote by $\T^d$ the box $\Omega$ with periodic boundary conditions. We have the following. 

\begin{lemm}\label{lemm:z}
We have
\bq\label{est:z}
\mez\frac{d}{dt}\Vert z(t)\Vert_{H^s(\T^d)}^2\le C_1 ( 1 + \Vert z(t)\Vert_{H^s(\T^d)} ) \Vert u(t)\Vert_{H^s(\T^d)}\Vert z(t)\Vert_{H^s(\T^d)}, 
\eq
for some constant $C_1$ depending only on $s, d$, and $\Vert y_*\Vert_{H^{s+1}}$. 
\end{lemm}
\begin{proof}
Denote $J^s=(1-\Delta)^\frac{s}{2}$. Note that $\Vert f\Vert_{H^s(\T^d)}=\Vert J^sf\Vert_{L^2(\T^d)}$. Applying $J^s$ to \eqref{eq:z}, multiplying the resulting equation by $J^sz$ and integrating in space we obtain
\begin{align*}
\mez\frac{d}{dt}\int_{\T^d} |J^sz|^2 \; dx&=-\int_{\T^d} \Big[ J^s z\cdot J^s(u\cdot \nabla z) + J^s(u\cdot \nabla y_*)\cdot J^sz\Big] \; dx\\
&=-\int_{\T^d} \Big[ J^sz\cdot \big([J^s,u]\cdot \nabla z\big)dx + J^sz\cdot \big(u\cdot \nabla J^sz\big) + J^s(u\cdot \nabla y_*)\cdot J^sz\Big] \; dx\\
&=-\int_{\T^d}\Big[ J^sz\cdot \big([J^s,u]\cdot \nabla z\big) + \frac12u\cdot \nabla |J^sz|^2dx + J^s(u\cdot \nabla y_*)\cdot J^sz \Big]\; dx\\
&=-\int_{\T^d}\Big[ J^sz\cdot \big([J^s,u]\cdot \nabla z\big) + J^s(u\cdot \nabla y_*)\cdot J^sz \Big]\; dx
\end{align*}
where we used the fact that $u=\P z$ is divergence free. Since $s>1+\frac{d}{2}$, the Kato-Ponce's commutator estimate yields
\[
\Vert [J^s,u]\cdot \nabla z\Vert_{L^2}\le C\Vert u\Vert_{H^s}\Vert z\Vert_{H^s}.
\]
On the other hand, we have 
\[
\la \int_{\T^d}J^s(u\cdot \nabla y_*)\cdot J^sz\ra\le C\Vert y_*\Vert_{H^{s+1}}\Vert u\Vert_{H^s}\Vert z\Vert_{H^s}.
\]
Putting this altogether leads to \eqref{est:z}.
\end{proof}

\begin{lemm}\label{lemm:u} Let $\theta_0$ be the constant as in \eqref{convexity}. There hold 
\begin{align}\label{L2decay}
\mez\frac{d}{dt}\Vert u(t)\Vert_{L^2(\T^d)}^2+ \theta_0\Vert u(t)\Vert_{L^2(\T^d)}^2 &\le C_2\Vert u(t)\Vert_{L^2(\T^d)}^2\Vert z(t)\Vert_{H^s(\T^d)} 
\\
\frac{d}{dt}\Vert u(t)\Vert_{H^s(\T^d)}^2+\theta_0\Vert u(t)\Vert_{H^s(\T^d)}^2 &\le C_2\Vert u(t)\Vert_{H^s(\T^d)}^2\Vert z(t)\Vert_{H^s(\T^d)}+C_2\Vert u(t)\Vert_{L^2(\T^d)}^2   \label{est:u1}
\end{align}
for some constant $C_2$ depending only on $s,d,$ and $\| y_*\|_{H^{s+1}}$.  
\end{lemm}
\begin{proof}
Applying the Leray projection to the first equation in \eqref{eq:z} gives
\bq\label{eq:u:0}
\p_t u+\P(u\cdot \nabla y_*)+\P(u\cdot \nabla z)=0.
\eq
A direct $H^s$ energy estimate for $u$ from this equation requires a control for $z$ in $H^{s+1}$, and thus the estimates cannot be closed. To resolve the issue, we first commute the  with $u$ to have
\bq\label{eq:u}
\p_tu+\P(u\cdot \nabla y_*)+u\cdot\nabla u+[\P, u\cdot\nabla] z=0
\eq
We now perform an $H^s$ energy estimate for this equation. As in the proof of Lemma \ref{lemm:z} we have
\[
\begin{aligned}
\mez&\frac{d}{dt}\int_{\T^d} |J^su|^2\; dx+\int_{\T^d}J^su\cdot J^s\P (u\cdot \nabla y_*) \; dx
\\&=-\int_{\T^d}J^su\cdot \big([J^s,u]\cdot \nabla u\big)dx-\int_{\T^d}J^su\cdot \big(u\cdot \nabla J^su\big)dx
-\int_{\T^d}J^su\cdot J^s\big([\P, u\cdot\nabla] z\big)dx\\
&=-\int_{\T^d}J^su\cdot \big([J^s,u]\cdot \nabla u\big)dx-\mez\int_{\T^d}u\cdot \nabla |J^su|^2dx
-\int_{\T^d}J^su\cdot J^s\big([\P, u\cdot\nabla] z\big)dx\\
&=-\int_{\T^d}J^su\cdot \big([J^s,u]\cdot \nabla u\big)dx-\int_{\T^d}J^su\cdot J^s\big([\P, u\cdot\nabla] z\big)dx
\end{aligned}
\]
where the fact that $u$ is divergence-free  was used to cancel out the term $\int_{\T^d}u\cdot \nabla |J^su|^2dx$. In addition, using the facts that $[J^s, \P]=0$ and $\P$ is self-adjoint in $L^2$, we obtain
\bq
\begin{aligned}\label{damping}
\int_{\T^d}J^su\cdot J^s\P (u\cdot \nabla y_*)&=\int_{\T^d}J^su\cdot \P J^s (u\cdot \nabla y_*) =\int_{\T^d}J^su\cdot J^s (u\cdot \nabla y_*)\\
&=\int_{\T^d}J^su\cdot (J^su \cdot \nabla y_*)+\int_{\T^d}J^su\cdot \big([J^s, \nabla y_*]\cdot u\big) .
\end{aligned}
\eq
It follows from the convexity assumption \eqref{convexityT} that 
\bq\label{low-Hs}
\int_{\T^d}J^su\cdot (J^su \cdot\nabla y_*)=\int_{\T^d}J^su\cdot (J^su \cdot\nabla^2 p_*)\ge \theta_0\Vert u\Vert_{H^s}^2.
\eq
Using Kato-Ponce's commutator estimate gives
\[
\Vert [J^s,u]\cdot \nabla u\Vert_{L^2}\le C\Vert u\Vert_{H^s}\Vert u\Vert_{H^s}
\]
and 
\[
\Vert [J^s, \nabla y_*]\cdot u\Vert_{L^2}\le C\Vert  \nabla y_*\Vert_{W^{1, \infty}}\Vert u\Vert_{H^{s-1}}+C\Vert  \nabla y_*\Vert_{H^s}\Vert u\Vert_{L^\infty}\le C\Vert y_*\Vert_{H^{s+1}}\Vert u\Vert_{H^{s-1}}.
\]
By virtue of  Theorem \ref{theo:cmt} for $\Omega=\T^d$ we have
\[
\Vert [\P, u\cdot\nabla] z\Vert_{H^s}\le C\Vert u\Vert_{H^s}\Vert z\Vert_{H^s}.
\]
Combining, and using $\|u\|_{H^s} \le \|z\|_{H^s}$, we thus obtain 
\bq\label{est:u}
\frac{d}{dt}\Vert u(t)\Vert_{H^s}^2+2\theta_0\Vert u(t)\Vert_{H^s}^2\le C_2\Vert u(t)\Vert_{H^s}^2\Vert z(t)\Vert_{H^s}+C_3\Vert  y_*\Vert_{H^{s+1}}\Vert u(t)\Vert_{H^s}\Vert u(t)\Vert_{H^{s-1}} 
\eq
in which using interpolation and Young's inequality, the last term can be estimated by 
\begin{align*}
\Vert u(t)\Vert_{H^s}\Vert u(t)\Vert_{H^{s-1}}&\le \Vert u(t)\Vert_{H^s}^{2-\frac{1}{s}}\Vert u(t)\Vert_{L^2}^{\frac{1}{s}} \le \gamma \Vert u(t)\Vert_{H^s}^2+C_\gamma \Vert u(t)\Vert_{L^2}^2 
\end{align*}
for $\gamma>0$. Taking $\gamma$ sufficiently small, we get the claimed $H^s$ estimates in \eqref{est:u1} directly from \eqref{est:u}.

 Finally, by means of the  Sobolev embedding $\|\nabla z \|_{L^\infty} \le C \| z\|_{H^s}$, the fact that $\P$ is self-adjoint in $L^2$, and the convexity assumption \eqref{convexity}, an $L^2$ energy estimate for \eqref{eq:u:0} gives the estimate \eqref{L2decay}. This ends the proof of the lemma. 
\end{proof}
For $L$-periodic smooth solutions $(u,z)$ defined on the maximal interval $[0, T^*)$ of \eqref{eq:z}, which exists thanks to the local existence theory in Theorem \ref{theo:lwp}, let us introduce the bootstrap norm
\bq\label{bootstrapnorm}
\mathcal{N}(t) : =\sup_{0\le \tau\le t}\Big(\Vert z(\tau)\Vert_{H^s(\T^d)}^2+ M^2 e^{\frac{\theta_0}{2}\tau}\Vert u(\tau)\Vert_{L^2(\T^d)}^2+ M e^{\frac{\theta_0}{2}\tau}\Vert u(\tau)\Vert_{H^s(\T^d)}^2\Big)
\eq
for some fixed and large $M>0$ and for $t<T^*$.
\begin{prop}\label{prop:bootstrap}
There exist positive constants $\eps, C_*$, depending only on $\theta_0$ and $\Vert y_*\Vert_{H^{s+1}}$, such that whenever $\mathcal{N}(0)<\eps$ we have $\mathcal{N}(t)\le C_*\mathcal{N}(0)$ for all $t<T^*$. 
\end{prop} 
\begin{proof}  We shall prove that 
\begin{equation}\label{claim} \mathcal{N}(t) \le C_0\mathcal{N}(0) + C_0 \mathcal{N}(t)^{3/2} \end{equation}
for all $t<T^*$. The proposition follows directly from the standard continuous induction.   

As for the claim \eqref{claim}, we integrate \eqref{est:z} in time and use the definition of $\mathcal{N}(t)$, yielding 
\begin{equation}\label{bd-zzz}
\begin{aligned}
 \| z(t)\|_{H^s}^2 
 &\le \| z(0)\|_{H^s}^2 + C_1 \int_0^t  ( 1 + \Vert z(\tau)\Vert_{H^s} ) \Vert u(\tau)\Vert_{H^s}\Vert z(\tau)\Vert_{H^s} \; d\tau
 \\&\le \| z(0)\|_{H^s}^2 + C_1M^{-1/2} (1 + \mathcal{N}(t)^{1/2}) \mathcal{N}(t) \int_0^t e^{-\theta_0 \tau/4} \; d\tau
 \\&\le \| z(0)\|_{H^s}^2 + C_3M^{-1/2} (1 + \mathcal{N}(t)^{1/2}) \mathcal{N}(t) . 
\end{aligned}\end{equation}
Next taking $M$ sufficiently large so that $\theta_0M>C_2$, we obtain from \eqref{L2decay} and \eqref{est:u1} that
$$
\frac{d}{dt} (M\Vert u(t)\Vert_{L^2}^2 + \| u(t)\|_{H^s}^2)+ \theta_0 (M\Vert u(t)\Vert_{L^2}^2 + \| u(t)\|_{H^s}^2) \le C_4M \Vert u(t)\Vert_{H^s}^2\Vert z(t)\Vert_{H^s} 
$$
where $C_4=C_2(2+M)$. This yields 
\bq\label{bt:u}
\begin{aligned}
M^2&\Vert u(t)\Vert_{L^2}^2 + M \| u(t)\|_{H^s}^2 
\\&\le 
e^{-\theta_0 t}\Big[ M^2\Vert u(0)\Vert_{L^2}^2 + M \| u(0)\|_{H^s}^2 \Big] 
+ C_4 M \int_0^t e^{-\theta_0 (t-\tau)} \Vert u(\tau)\Vert_{H^s}^2\Vert z(\tau)\Vert_{H^s} \; d\tau 
\\
&\le e^{-\theta_0 t}\Big[ M^2\Vert u(0)\Vert_{L^2}^2 + M \| u(0)\|_{H^s}^2 \Big]  + C_4 \mathcal{N}(t)^{3/2}  \int_0^t e^{-\theta_0 (t-\tau)} e^{-\theta_0\tau/2} \; d\tau 
\\
&\le e^{-\theta_0 t/2}\Big[ M^2\Vert u(0)\Vert_{L^2}^2 + M \| u(0)\|_{H^s}^2 \Big]  + C_5 \mathcal{N}(t)^{3/2} e^{-\theta_0 t/2}.\end{aligned}
\eq
Combining with \eqref{bd-zzz} and choosing again $M$ large, if needed, we obtain the claim \eqref{claim} and hence the proposition. 
\end{proof}
 With the $\eps$ and $C_*$ given in Proposition \ref{prop:bootstrap}, we have proven that $\Vert z(t)\Vert_{H^s(\T^d)}\le C_*\mathcal{N}(0)\le C_*\eps$ and 
$\Vert u(t)\Vert_{H^s(\T^d)} \le C\Vert y_0-y_*\Vert_{H^s(\T^d)}e^{-\frac{\theta_0}{4} t}$ for all time $t <T^*$. Consequently, the solution $z$ of \eqref{eq:z} is global in time and enjoys the same bounds for all $t>0$. Using equation \eqref{eq:z} and the estimates \eqref{stab:est}, \eqref{damping:u} we deduce that $\p_t z\in L^1(0, \infty; H^{s-1}(\T^d))$. This yields
\[
\lim_{t\to\infty} z(x, t)=z_0(x)+\int_0^\infty \p_tz(x, \tau)d\tau:=z_\infty(x) \quad\text{in}~H^{s-1}(\T^d),
\]
and thus 
\[
\lim_{t\to\infty} y(x, t)=y_\infty(x):=z_\infty(x)+y_*(x)\quad\text{in}~H^{s-1}(\Omega),\quad \Omega=[0, L]^d.
\]
Furthermore, 
\bq\label{exp:conv}
\| y(t)-y_\infty\|_{H^{s-1}(\T^d)}=\| z(t)-z_\infty\|_{H^{s-1}(\T^d)}=\| \int_t^\infty u\cdot \na z(\tau)d\tau\|_{H^{s-1}(\T^d)}\le Ce^{-\frac{\tt_0t}{4}}
\eq
for all $t\ge 0$. Using the Leray projection we write 
\[
y(x, t)=u(x, t)+\nabla p(x, t),\quad y_\infty(x)=u_\infty(x)+\nabla p_\infty(x)
\]
 where $u_\infty=\P y_\infty:\Omega\to \Rr^d$ and $p:\Omega \to \Rr$. In view of the Pythagorean identity 
 \[
\| y(t)-y_\infty\|_{L^2(\Omega)}^2=\| u(t)-u_\infty\|_{L^2(\Omega)}^2+\| \nabla p(t)-\nabla p_\infty\|_{L^2(\Omega)}^2
\]
we find that each term on the right hand side converges to $0$ as $t\to \infty$. This, together with the fact that $u(t)\to 0$ in $H^s(\T^d)$, implies that $u_\infty\equiv 0$. Thus, $y_\infty=\nabla p_\infty$ is a gradient and in view of \eqref{exp:conv} we have 
\[
\| y(t)-\na p_\infty\|_{H^{s-1}(\T^d)}\le Ce^{-\frac{\tt_0t}{4}}
\]
 for all $t\ge 0$. As a consequence of this, \eqref{convexityT} and the bound $\| y-y_*\|_{L^\infty(0, \infty; H^s(\T^d))}\le C_*\eps$, if $\eps$ is sufficiently small then $\na^2 p_\infty >0$. Thus, $p_\infty$ is (strictly) convex and $\na p_\infty$ is the optimal rearrangement of $y_0$ by virtue of Brenier's theorem (\cite{Brenier91}). This ends the proof of Theorem \ref{theo:globalT}. 

\section{Proof of Theorem \ref{theo:global}}
\label{sec-proofB}
We now turn to the case when $\Omega$ is a bounded domain with smooth boundary. Recall that the perturbation $z=y-y_*$ obeys 
\bq\label{eq:z:O}
\begin{aligned}
\p_tz+u\cdot\nabla y_*+u\cdot \nabla z&=0,\qquad
u=\P z.
\end{aligned}
\eq 
In what follows, we fix an integer $s>1+\frac d2$.
\begin{lemm}
There exists $C>0$ depending only on $\| y_*\|_{H^{s+1}(\Omega)}$ such that
\bq\label{est:z:O}
\mez\frac{d}{dt}\Vert z(t)\Vert_{H^s(\Omega)}^2\le C_1 ( 1 + \Vert z(t)\Vert_{H^s(\Omega)} ) \Vert u(t)\Vert_{H^s(\Omega)}\Vert z(t)\Vert_{H^s(\Omega)}.
\eq
\end{lemm}
\begin{proof}
First of all, an $L^2$ estimate for \eqref{eq:z:O} gives
\bq\label{est:z0O}
\mez\frac{d}{dt}\Vert z(t)\Vert_{L^2(\Omega)}^2= -\int_\Omega z\cdot (u\cdot \na y_*)dx\le \| \na y_*\|_{L^\infty(\Omega)}\| u\|_{L^2(\Omega)}\| z\|_{L^2(\Omega)}, 
\eq
where we used the fact that $u\cdot n\vert_{\p\Omega}=0$ to have $\int_\Omega z\cdot (u\cdot \na z)=0$ upon integration by parts. Now let $s$ be an integer greater than $1+d/2$. Recall that $D^s$ denotes any partial derivatives of order $s$. Applying $D^s$ to equation \eqref{eq:z} and arguing as in  the proof of Lemma \ref{lemm:z} yields 
\bq\label{est:zO}
\mez\frac{d}{dt}\Vert D^sz(t)\Vert_{L^2(\Omega)}^2\le C_1 ( 1 + \Vert z(t)\Vert_{H^s(\Omega)} ) \Vert u(t)\Vert_{H^s(\Omega)}\Vert z(t)\Vert_{H^s(\Omega)}, 
\eq
upon using the  commutator estimate (see \cite{MajBer} page 129)
\bq
\| D^s(fg)-fD^sg\|_{L^2(\Omega)}\le C\| \na f\|_{L^\infty(\Omega)}\| g\|_{H^{s-1}(\Omega)}+C\| f\|_{H^s(\Omega)}\|g\|_{L^\infty(\Omega)} 
\eq
and the fact that $u\cdot n\vert_{\p\Omega}=0$  when taking integration by parts in integral  involving the highest derivatives $u \cdot \nabla |D^s z|^2$. Combining \eqref{est:z0O} and \eqref{est:zO} leads to the estimate \eqref{est:z:O}.
\end{proof}
Next let us recall equation \eqref{eq:u:0} for $u$
\bq\label{eq:u:1}
\p_t u+\P(u\cdot \nabla y_*)+\P(u\cdot \nabla z)=0.
\eq
Repeating verbatim the proof of \eqref{L2decay} we obtain 
\bq\label{u:L2:O}
\mez\frac{d}{dt}\Vert u(t)\Vert_{L^2(\Omega)}^2+ \theta_0\Vert u(t)\Vert_{L^2(\Omega)}^2 \le C_2\Vert u(t)\Vert_{L^2(\Omega)}^2\Vert z(t)\Vert_{H^s(\Omega)}.
\eq
We will need decay of the $H^s$ norm of $u$. Let us note that the proof of \eqref{est:u1} in Lemma \ref{lemm:u} does carry over to domains with boundary since the Leray projector does not commute with $D^s$, as used in \eqref{damping} for $\Omega=\T^d$.  To treat the boundary and the nonlocality of $\P$, we use the derivatives $\p_{\tau_j}$ and $\p_n$ introduced in Section \ref{sec:tn}. These derivatives are defined everywhere in $\Omega$ and become the usual tangential and normal derivative when restricted to the boundary. The trade-off is that $\p_{\tau_j}$ and $\p_n$ do not commute with usual partial derivatives, leading  to commutators that are of lower order.

For $k\in \{0, 1, .., s\}$ we set 
\[
\cD^s_k=\left\{\Pi_{j=1}^s\p_{\sigma_j}: \sigma_j\in \{\tau_1,...,\tau_{d-1}, n\}~\text{and}~\#\{j: \sigma_j=n\}=k\right\}.
\]
In other words, each derivative in $\cD^s_k$ has exactly $k$ normal derivatives and $s-k$ tangential derivatives. We also define the norms
\[
\| v\|_{s, k}=\Big(\sum_{j=0}^k\sum_{ P \in \cD^s_j} \|  Pv\|^2_{L^2(\Omega)}\Big)^\mez
\]
for $v:\Omega\to \Rr^d$. 

Due to the presence of  $\chi_1$ in $\p_n$ and $\p_{\tau_j}$, the norms $\| u\|_{s, k}$ control $u$ near the boundary. 

\subsection{Interior estimates for $u$}
The next lemma provides a control of $u$ in the interior.
\begin{lemm}\label{lemm:interior}
There exists $C>0$ depending only on $\| y_*\|_{H^{s+1}(\Omega)}$ such that
\bq\label{est:interior}
\begin{aligned}
&\mez\frac{d}{dt} \|\chi_2 u\|_{H^s}^2+\theta_0 \|\chi_2 u\|_{H^s}^2\le C\Vert u\Vert_{H^s(\Omega)}^2\Vert z\Vert_{H^s(\Omega)}+C\Vert u\Vert_{H^{s-1}(\Omega)}\Vert u\Vert_{H^s(\Omega)}
\end{aligned}
\eq
where $\chi_2$ is defined in \eqref{def:chi2}.
\end{lemm}
\begin{proof}
As in \eqref{eq:u:0}, we commute $\P$ with $u\cdot \na$ in the last term of equation \eqref{eq:u:1} to have 
\bq\label{eq:u:2}
\p_t u+\P(u\cdot \nabla y_*)+u\cdot \na u+[\P, u\cdot \nabla] z=0.
\eq
Set  $P=\chi_2 \p_1^s$.  Applying  $P$ to \eqref{eq:u:2}, then multiplying the resulting equation by $Pu$ and integrating over $\Omega$, we obtain
\bq\label{est:int:dt}
\begin{aligned}
\mez&\frac{d}{dt}\int_{\Omega} |Pu|^2\; dx+\int_{\Omega}Pu\cdot P\P (u\cdot \nabla y_*) \; dx\\
&=-\int_{\Omega}Pu\cdot \big([P,u]\cdot \nabla u\big)dx-\int_{\Omega}Pu\cdot \big(u\cdot  [P, \na]u\big)dx -\int_{\Omega}Pu\cdot P\big([\P, u\cdot\nabla] z\big)dx.
\end{aligned}
\eq
where we used the fact that 
\[
\int_{\Omega}Pu\cdot \big(u\cdot Pu\big)dx=\mez\int_{\Omega}u\cdot \nabla |Pu|^2dx=0
\]
since $\na\cdot u=0$ in $\Omega$ and $u\cdot n\vert_{\p\Omega}=0$. We now treat each term on the right-hand side of \eqref{est:int:dt}. It is readily seen that
\bq\label{bound:dtan1}
\begin{aligned}
&\| [P,u]\cdot \nabla u\|_{L^2(\Omega)}\le C\| u\|_{H^s(\Omega)}^2,\\
& \| u\cdot [P, \na ] u\|_{L^2(\Omega)}\le C\| u\|^2_{H^{s-1}(\Omega)}.
\end{aligned}
\eq
In addition, Theorem \ref{theo:cmt} applied to $\Omega$ gives 
\bq\label{bound:dtan2}
\| P\big([\P, u\cdot\nabla] z\big)\|_{L^2(\Omega)}\le C\|[\P, u\cdot\nabla] z\|_{H^s(\Omega)}\le C\| u\|_{H^s(\Omega)}\| z\|_{H^s(\Omega)}.
\eq
Putting together \eqref{est:int:dt}, \eqref{bound:dtan1}, \eqref{bound:dtan2} and using the estimate $\| u\|_{H^s}\le C\|z\|_{H^s}$ we obtain
\bq\label{cmt:t:est}
\begin{aligned}
&\mez&\frac{d}{dt}\int_{\Omega} |Pu|^2\; dx+\int_{\Omega}Pu\cdot P\P (u\cdot \nabla y_*) \; dx\le C\Vert u\Vert_{H^s(\Omega)}^2\Vert z\Vert_{H^s(\Omega)}.
\end{aligned}
\eq
As for the second term on the left-hand side of \eqref{cmt:t:est}, we commute $P$ with $\P$ and then with $\na y_*$ to have
\bq\label{cmt:t:damp}
\begin{aligned}
\int_{\Omega}Pu\cdot P\P (u\cdot \nabla y_*) \; dx&=\int_{\Omega}Pu\cdot [P,\P] (u\cdot \nabla y_*) \; dx+\int_{\Omega}Pu\cdot \P ([P, \nabla y_*\cdot] u) \; dx\\
&\quad +\int_{\Omega}Pu\cdot  \P (\nabla y_*\cdot P u) \; dx\\
&=\int_{\Omega}Pu\cdot [P,\P] (u\cdot \nabla y_*) \; dx+\int_{\Omega}Pu\cdot \P ([P, \nabla y_*\cdot] u) \; dx\\
&\quad +\int_{\Omega}[\P, P]u\cdot  (\nabla y_*\cdot P u) \; dx+\int_{\Omega} Pu\cdot  (\nabla y_*\cdot P u) \; dx.
\end{aligned}
\eq
By virtue of Proposition \ref{cmt:Pna},
\[
\| [\P, P]u\|_{H^s(\Omega)}\le C\| u\|_{H^{s-1}(\Omega)}
\]
and
\[
\| [P,\P] (u\cdot \nabla y_*) \|_{L^2(\Omega)}\le C\| u\cdot \nabla y_*\|_{H^{s-1}(\Omega)}\le C\| u\|_{H^{s-1}(\Omega)}\| y_*\|_{H^s(\Omega)}.
\]
The local commutator $[P, \nabla y_*\cdot] u$ can be bounded as
\bq\label{cmt:t:y*}
\|[P, \nabla y_*\cdot ] u\|_{L^2(\Omega)}\le C\| y_*\|_{H^{s+1}(\Omega)}\|u\|_{H^{s-1}(\Omega)}.
\eq
On the other hand, the convexity condition \eqref{convexity} yields
\[
\int_{\Omega}Pu\cdot  (\nabla y_*\cdot P u) \; dx\ge \theta_0 \| P u\|_{L^2(\Omega)}^2.
\]
We then deduce from \eqref{cmt:t:damp} that
\bq
\int_{\Omega}Pu\cdot P\P (u\cdot \nabla y_*) \; dx\ge \theta_0 \| P u\|_{L^2(\Omega)}^2-C\| y_*\|_{H^{s+1}(\Omega)}\|u\|_{H^{s-1}(\Omega)}\|u\|_{H^s(\Omega)}
\eq
which combined with \eqref{cmt:t:est} leads to \eqref{est:tang}. The same estimates hold for mixed derivatives $\chi_2D^s$ where $D^s$ is any partial derivative of order $s$. 
\end{proof}
\subsection{Estimates for tangential derivatives of $u$}

\begin{lemm}\label{lemm:tang}
There exists $C>0$ depending only on $\| y_*\|_{H^{s+1}(\Omega)}$ such that
\bq\label{est:tang}
\begin{aligned}
&\mez\frac{d}{dt} \|u\|_{s, 0}^2+\theta_0 \| u\|_{s, 0}^2\le C\Vert u\Vert_{H^s(\Omega)}^2\Vert z\Vert_{H^s(\Omega)}+C\Vert u\Vert_{H^{s-1}(\Omega)}\Vert u\Vert_{H^s(\Omega)}.
\end{aligned}
\eq
\end{lemm}
\begin{proof}
The proof follows along the same lines as in Lemma \ref{lemm:interior} upon taking $P\in \cD^s_0$ and using Proposition \ref{prop:estuO} in place of Proposition \ref{cmt:Pna}.
\end{proof}
\subsection{Estimates for mixed derivatives of $u$}
 The next lemma concerns $\| u\|_{s, 1}$.
\begin{lemm}\label{lemm:mix:1n}
There exists $M_1>0$ such that
\bq\label{est:mix:1n}
\begin{aligned}
\mez\frac{d}{dt}\| u\|^2_{s, 1}+\frac{\theta_0}{2} \| u\|^2_{s, 1}&\le M_1\Vert u\Vert_{H^s(\Omega)}^2\Vert z\Vert_{H^s(\Omega)}+M_1\Vert u\Vert_{H^{s-1}(\Omega)}\Vert u\Vert_{H^s(\Omega)} +M_1\| u\|_{s, 1}\| u\|_{s, 0}.
\end{aligned}
\eq
\end{lemm}
\begin{proof}
Let $P\in \cD_1^s$. Assume without loss of generality that $P=\p_{\tau_1}^{s-1}\p_n$.  Commuting equation \eqref{eq:u:2} with $P$ gives
\bq
\begin{aligned}
\mez&\frac{d}{dt}\int_{\Omega} |Pu|^2\; dx+\int_{\Omega}Pu\cdot P\P (u\cdot \nabla y_*) \; dx\\
&=-\int_{\Omega}Pu\cdot \big([P,u]\cdot \nabla u\big)dx-\int_{\Omega}Pu\cdot \big(u\cdot  [P, \na]u\big)dx -\int_{\Omega}Pu\cdot P\big([\P, u\cdot\nabla] z\big)dx.
\end{aligned}
\eq
Arguing as in the proof of Lemma \ref{lemm:tang}, we find that the right-hand side is bounded by $C\Vert u\Vert_{H^s(\Omega)}^2\Vert z\Vert_{H^s(\Omega)}$. 
Now we write using the definition of $\P$ that 
\[
\int_{\Omega}Pu\cdot P\P (u\cdot \nabla y_*) \; dx=\int_{\Omega}Pu\cdot P (u\cdot \nabla y_*) \; dx-\int_{\Omega}Pu\cdot P\na f \; dx
\]
where $f$ solves
\bq\label{eq:f:mix0}
\begin{cases}
\Delta f=\dive(u\cdot \na y_*)\quad\text{in}~\Omega,\\
\p_nf=(u\cdot \nabla y_*)\cdot n\quad\text{on}~\p\Omega.
\end{cases}
\eq
Commuting $P$ with $\na y_*$ gives
\[
\int_{\Omega}Pu\cdot P (u\cdot \nabla y_*) \; dx=\int_{\Omega}Pu\cdot (P u\cdot \nabla y_*) \; dx+\int_{\Omega}Pu\cdot [P, \nabla y_*\cdot ] u\; dx
\]
where the local commutator $[P, \nabla y_*\cdot] u$ satisfies
\[
\| [P, \nabla y_*]\cdot u\|_{L^2(\Omega)}\le C\| y_*\|_{H^{s+1}(\Omega)}\| u\|_{H^{s-1}(\Omega)}
\]
and by the convexity assumption \eqref{convexity}, 
\[
\int_{\Omega}Pu\cdot (P u\cdot \nabla y_*) \; dx\ge \theta_0\| Pu\|_{L^2(\Omega)}^2.
\]
The rest of this proof is devoted to the control of $\int_{\Omega}Pu\cdot P\na f \; dx$.  First, since $\chi_1\equiv 1$ in $\Omega_{2\kappa}\supset\supp (1-\chi_2)$, in view of \eqref{decompose:na1}, the decomposition 
\bq\label{decompose:na}
\na g=(1-\chi_2)n\p_ng+(1-\chi_2)\tau_j\p_{\tau_j}g+\chi_2\na g
\eq
holds in $\Omega$ for any scalar $g:\Omega\to \Rr$. 
Using this with $g=Pf$, we write
\bq
\begin{aligned}
\int_{\Omega}Pu\cdot P\na f dx&=\int_{\Omega}Pu\cdot \na Pf dx+\int_{\Omega}Pu\cdot [P, \na] fdx\\
&=\int_{\Omega}(1-\chi_2)(Pu\cdot n)\p_nPf dx+\int_{\Omega}(1-\chi_2)(Pu\cdot \tau_j)\p_{\tau_j}Pf dx\\
&\quad +\int_{\Omega}\chi_2 Pu\cdot \na Pf dx+\int_{\Omega}Pu\cdot [P, \na] fdx\\
&=I_1+I_2+I_3+I_4.
\end{aligned}
\eq
Due to the presence of the local commutator $[P, \na] f$, it is readily seen that 
\bq\label{I4:1n}
| I_4 |\le C\| u\|_{H^s}\| f\|_{H^{s}(\Omega)}\le C'\| u\|_{H^s}\| u\|_{H^{s-1}(\Omega)}.
\eq
As for $I_3$, we integrate by parts noticing that $\dive u=0$ in $\Omega$ and $\chi_2\equiv 0$ near $\p\Omega$ to obtain 
\bq
\begin{aligned}
I_3=\int_{\Omega}\chi_2 Pu\cdot \na Pf dx
=-\int_{\Omega}[\dive, \chi_2 P]u Pf 
\end{aligned}
\eq
which implies
\bq\label{I3:1n}
|I_3|\le C\|u\|_{H^s(\Omega)}\| u\|_{H^{s-1}(\Omega)}. 
\eq
{\it Estimate for $I_1$.} We first note that
\[
\begin{aligned}
Pu\cdot n&=\p_{\tau_1}^{s-1}(\p_n u\cdot n)-[\p_{\tau_1}^{s-1}, n\cdot]\p_nu\\
&=-\p_{\tau_1}^{s-1}(\p_{\tau_j}u\cdot \tau_j)-[\p_{\tau_1}^{s-1}, n\cdot]\p_nu\\
&=-(\p_{\tau_1}^{s-1}\p_{\tau_j}u)\cdot \tau_j-[\p_{\tau_1}^{s-1}, \tau_j]\p_{\tau_j}u-[\p_{\tau_1}^{s-1}, n\cdot]\p_nu
\end{aligned}
\]
This implies
\bq\label{1nu.n}
\| Pu\cdot n\|_{L^2(\Omega)}\le  C\| u\|_{s, 0}+C\| u\|_{H^{s-1}(\Omega)}.
\eq
On the other hand, it follows from \eqref{eq:f:mix0} that
\bq\label{elliptic:1n}
\begin{cases}
\Delta \p_{\tau_1}^{s-1}f=[\Delta, \p_{\tau_1}^{s-1}]f+[\p_{\tau_1}^{s-1},\dive](u\cdot \na y_*)+\dive \p_{\tau_1}^{s-1}(u\cdot \na y_*):=g_1\quad\text{in}~\Omega,\\
\p_n\p_{\tau_1}^{s-1}f=[\p_n, \p_{\tau_1}^{s-1}]f+\p_{\tau_1}^{s-1}\{(u\cdot \nabla y_*)\cdot n\}:=g_2\quad\text{on}~\p\Omega.
\end{cases}
\eq
It is easy to see that 
\[
\| [\Delta, \p_{\tau_1}^{s-1}]f\|_{L^2(\Omega)}+[\p_{\tau_1}^{s-1},\dive](u\cdot \na y_*)\|_{L^2(\Omega)}\le C\| u\|_{H^{s-1}(\Omega)}.
\]
In addition, \eqref{nav:tn} gives
\[
\begin{aligned}
&\|\dive \p_{\tau_1}^{s-1}(u\cdot \na y_*)\|_{L^2(\Omega)}\\
&\le C\| \na \p_{\tau_1}^{s-1}(u\cdot \na y_*)\|_{L^2(\Omega)}\\
&\le C\| \p_n \p_{\tau_1}^{s-1}(u\cdot \na y_*)\|_{L^2(\Omega)}+C\| \p_{\tau_j} \p_{\tau_1}^{s-1}(u\cdot \na y_*)\|_{L^2(\Omega)}\\
&\le   C\| u\|_{s, 1}+C\|u\|_{s, 0}.
\end{aligned}
\]
Consequently 
\[
\| g_1\|_{L^2(\Omega)}\le C\|u\|_{H^{s-1}(\Omega)}+ C\|u\|_{s, 0}+C\| u\|_{s, 1}.
\]
Using the trace inequality and arguing as above we obtain that
\[
\| g_2\|_{H^\mez(\p\Omega)}\le C\|u\|_{H^{s-1}(\Omega)}+ C\|u\|_{s, 0}+C\| u\|_{s, 1}.
\]
Then the $H^2$ elliptic estimate for \eqref{elliptic:1n} leads to
\bq\label{H2:1nu}
\| \p_{\tau_1}^{s-1}f\|_{H^2(\Omega)}\le C\| g_1\|_{L^2(\Omega)}+C\| g_2\|_{H^\mez(\p\Omega)}\le C\|u\|_{H^{s-1}(\Omega)}+ C\|u\|_{s, 0}+C\| u\|_{s, 1}.
\eq
Next we write 
\[
\begin{aligned}
I_1&=\int_{\Omega}(1-\chi_2)(Pu\cdot n)\p_nPf dx\\
&=\int_{\Omega}(1-\chi_2)(Pu\cdot n)P\p_nf dx+\int_{\Omega}(1-\chi_2)(Pu\cdot n)[\p_n, P]f dx\\
&=\int_{\Omega}(1-\chi_2)(Pu\cdot n)\p_{\tau_1}^{s-1}\p^2_nf dx+\int_{\Omega}(1-\chi_2)(Pu\cdot n)[\p_n, P]f dx\\
&=\int_{\Omega}(1-\chi_2)(Pu\cdot n)\p^2_n\p_{\tau_1}^{s-1}f dx+\int_{\Omega}(1-\chi_2)(Pu\cdot n)[\p_{\tau_1}^{s-1},\p^2_n]f dx\\
&\quad+\int_{\Omega}(1-\chi_2)(Pu\cdot n)[\p_n, P]f dx.
\end{aligned}
\]
In view of \eqref{1nu.n} and \eqref{H2:1nu} we deduce that 
\bq\label{I1:1n}
|I_1|\le  C\|u\|_{H^{s-1}(\Omega)}^2+ C\|u\|^2_{s, 0}+C\| u\|_{s, 1}\|u\|_{s, 0}.
\eq
{\it Estimate for $I_2$.} We first write 
\[
\begin{aligned}
I_2&=\int_{\Omega}(1-\chi_2)(Pu\cdot \tau_j)\p_{\tau_j}\p_{\tau_1}^{s-1}\p_nf dx\\
&=\int_{\Omega}(1-\chi_2)(Pu\cdot \tau_j)\p_n\p_{\tau_j}\p_{\tau_1}^{s-1}f dx+\int_{\Omega}(1-\chi_2)(Pu\cdot \tau_j)[\p_{\tau_j}\p_{\tau_1}^{s-1}, \p_n]f dx
\end{aligned}
\]
where 
\bq\label{1n:I2:1}
\la \int_{\Omega}(1-\chi_2)(Pu\cdot \tau_j)[\p_{\tau_j}\p_{\tau_1}^{s-1}, \p_n]f dx\ra\le C\| u\|_{H^s(\Omega)}\|u\|_{H^{s-1}(\Omega)}.
\eq
On the other hand, by H\"older's and Young's inequality,
\bq\label{1n:I2:2}
\begin{aligned}
\la \int_{\Omega}(1-\chi_2)(Pu\cdot \tau_j)\p_n\p_{\tau_j}\p_{\tau_1}^{s-1}f dx\ra&\le C\| Pu\cdot \tau_j\|_{L^2(\Omega)}\|\p_n \p_{\tau_j}\p_{\tau_1}^{s-1}f\|_{L^2(\Omega)}\\
&\le C\| P u\|_{L^2(\Omega)}\| \p_{\tau_j}\p_{\tau_1}^{s-1}f\|_{H^1(\Omega)}\\
&\le \frac{\tt_0}{2} \| P u\|^2_{L^2(\Omega)}+C'\| \p_{\tau_j}\p_{\tau_1}^{s-1}f\|_{H^1(\Omega)}^2.
\end{aligned}
\eq
Using again equation \eqref{eq:f:mix0} we find 
\bq\label{elliptic:1n:2}
\begin{cases}
\Delta \p_{\tau_j}\p_{\tau_1}^{s-1}f=[\Delta, \p_{\tau_j}\p_{\tau_1}^{s-1}]f+[\p_{\tau_j}\p_{\tau_1}^{s-1},\dive](u\cdot \na y_*)+\dive \p_{\tau_j}\p_{\tau_1}^{s-1}(u\cdot \na y_*)\quad\text{in}~\Omega,\\
\p_n\p_{\tau_j}\p_{\tau_1}^{s-1}f=[\p_n, \p_{\tau_j}\p_{\tau_1}^{s-1}]f+\p_{\tau_j}\p_{\tau_1}^{s-1}\{(u\cdot \nabla y_*)\cdot n\}\quad\text{on}~\p\Omega.
\end{cases}
\eq
Multiplying the first equation by $\p_{\tau_j}\p_{\tau_1}^{s-1}f$ then integrating over $\Omega$ and using the second equation to cancel out the leading boundary term, we deduce that $h=\p_{\tau_j}\p_{\tau_1}^{s-1}f$ satisfies
\[
\begin{aligned}
\int_\Omega |\na h|^2dx&=-\int_\Omega h\Big\{[\Delta, \p_{\tau_j}\p_{\tau_1}^{s-1}]f+[\p_{\tau_j}\p_{\tau_1}^{s-1},\dive](u\cdot \na y_*)\Big\}dx+\int_{\p\Omega}h [\p_n, \p_{\tau_j}\p_{\tau_1}^{s-1}]fdS\\
&=I_{2a}+I_{2b}.
\end{aligned}
\]
We observe that $\| h\|_{L^2(\Omega)}\le C\|u\|_{H^{s-1}(\Omega)}$ and 
\[
\| [\Delta, \p_{\tau_j}\p_{\tau_1}^{s-1}]f\|_{L^2(\Omega)}+[\p_{\tau_j}\p_{\tau_1}^{s-1},\dive](u\cdot \na y_*)\|_{L^2(\Omega)}\le C\| u\|_{H^s(\Omega)},
\]
hence 
\[
|I_{2a}|\le C\| u\|_{H^s(\Omega)}\| u\|_{H^{s-1}(\Omega)}.
\] 
On the other hand, by virtue of the trace inequality and interpolation, the surface integral is controlled as
\[
\begin{aligned}
|I_{2b}|\le C\| f\|_{H^s(\p\Omega)}^2\le C'\| f\|^2_{H^{s+\mez}(\Omega)}\le C''\| f\|_{H^s(\Omega)}\| f\|_{H^{s+1}(\Omega)}\le C'''\| u\|_{H^{s-1}(\Omega)}\| u\|_{H^s(\Omega)}.
\end{aligned}
\]
It follows that
\[
\| h\|_{H^1(\Omega)}^2 \le \| h\|_{L^2(\Omega)}^2+\|\na h\|_{L^2(\Omega)}^2\le C\| u\|_{H^s(\Omega)}\| u\|_{H^{s-1}(\Omega)}.
\]
Plugging this into \eqref{1n:I2:2} and recalling \eqref{1n:I2:1} we deduce that 
\bq
|I_2|\le \frac{\tt_0}{2} \| P u\|^2_{L^2(\Omega)}+C\| u\|_{H^s(\Omega)}\|u\|_{H^{s-1}(\Omega)}.
\eq
Putting together the above considerations  we arrive at
\[
\mez\frac{d}{dt}\| Pu\|^2_{L^2}+\frac{\tt_0}{2}\| Pu\|_{L^2}\le C\Vert u\Vert_{H^s(\Omega)}^2\Vert z\Vert_{H^s(\Omega)}+C\Vert u\Vert_{H^s(\Omega)}\Vert u\Vert_{H^{s-1}(\Omega)}.
\]
Then summing over all $P\in \cD^s_1$ yields\[
\mez\frac{d}{dt}\| u\|_{s, 1}^2+\frac{\tt_0}{2}\| u\|_{s, 1}^2\le M_1\Vert u\Vert_{H^s(\Omega)}^2\Vert z\Vert_{H^s(\Omega)}+M_1\Vert u\Vert_{H^s(\Omega)}\Vert u\Vert_{H^{s-1}(\Omega)}
\]
which combined with \eqref{est:tang} for tangential derivatives leads to the desired estimate \eqref{est:mix:1n}. 
\end{proof}
\begin{lemm}\label{lemm:mix:kn}
For each  $k\in \{1, 2, ...,s\}$ there exists  $M_k>0$ such that
\bq\label{est:mix:k}
\begin{aligned}
\mez\frac{d}{dt}\| u\|^2_{s, k}+\frac{\theta_0}{2} \| u\|^2_{s, k}&\le M_k\Vert u\Vert_{H^s(\Omega)}^2\Vert z\Vert_{H^s(\Omega)}+M_k\Vert u\Vert_{H^{s-1}(\Omega)}\Vert u\Vert_{H^s(\Omega)} +M_k\| u\|_{s, k}\| u\|_{s, k-1}.
\end{aligned}
\eq
\end{lemm}
\begin{proof}
  The base case $k=1$ has been proved in Lemma \ref{lemm:mix:1n}. Assume  \eqref{est:mix:k} for some $k\in \{1, 2,...,s-1\}$ we prove it for $k+1$ in place of $k$. Let $P\in \cD^s_{k+1}$. We assume without loss of generality that $P=\p_{\tau_1}^{s-k-1}\p_n^{k+1}$. Commuting equation \eqref{eq:u:2} with $P$ gives
\bq
\begin{aligned}
\mez&\frac{d}{dt}\int_{\Omega} |Pu|^2\; dx+\int_{\Omega}Pu\cdot P\P (u\cdot \nabla y_*) \; dx\\
&=-\int_{\Omega}Pu\cdot \big([P,u]\cdot \nabla u\big)dx-\int_{\Omega}Pu\cdot \big(u\cdot  [P, \na]u\big)dx -\int_{\Omega}Pu\cdot P\big([\P, u\cdot\nabla] z\big)dx.
\end{aligned}
\eq
As in the proof of Lemma \ref{lemm:mix:1n} it suffices to 
 treat the damping term 
 \[
\int_{\Omega}Pu\cdot P\P (u\cdot \nabla y_*) \; dx=\int_{\Omega}Pu\cdot P (u\cdot \nabla y_*) \; dx-\int_{\Omega}Pu\cdot P\na f \; dx
\]
where $f$ solves \eqref{eq:f:mix0}:
\bq\label{eq:f:mix}
\begin{cases}
\Delta f=\dive(u\cdot \na y_*)\quad\text{in}~\Omega,\\
\p_nf=(u\cdot \nabla y_*)\cdot n\quad\text{on}~\p\Omega.
\end{cases}
\eq
Commuting $P$ with $\na y_*$ gives
\[
\int_{\Omega}Pu\cdot P (u\cdot \nabla y_*) \; dx=\int_{\Omega}Pu\cdot (P u\cdot \nabla y_*) \; dx+\int_{\Omega}Pu\cdot [P, \nabla y_*]\cdot u\; dx
\]
where the local commutator $[P, \nabla y_*]\cdot u$ satisfies
\[
\| [P, \nabla y_*]\cdot u\|_{L^2(\Omega)}\le C\| y_*\|_{H^{s+1}(\Omega)}\| u\|_{H^{s-1}(\Omega)},
\]
and by the convexity assumption \eqref{convexity}, 
\[
\int_{\Omega}Pu\cdot (P u\cdot \nabla y_*) \; dx\ge \theta_0\| Pu\|_{L^2(\Omega)}^2.
\]
Then it remains to prove that
\bq\label{est:mix:main}
\int_{\Omega}Pu\cdot P\na f \; dx\le C\| u\|_{H^{s}(\Omega)}\| u\|_{H^{s-1}(\Omega)}+C\| u\|_{s, k+1}\| u\|_{s, k}.
\eq
 To this end, let us write using the decomposition \eqref{decompose:na} that for $k\ge 1$,
\[
\begin{aligned}
P\na f &=\na \p_{\tau_1}^{s-k-1}\p_n^{k-1}\p_n^2f+[\p_{\tau_1}^{s-k-1}\p_n^{k+1}, \na ]f\\
&=(1-\chi_2)\na \p_{\tau_1}^{s-k-1}\p_n^{k-1}\p_n^2f+\chi_2\na \p_{\tau_1}^{s-k-1}\p_n^{k-1}\p_n^2f+[\p_{\tau_1}^{s-k-1}\p_n^{k+1}, \na ]f.
\end{aligned}
\]
The commutator is a lower order term in the sense that
\[
\| [\p_{\tau_1}^{s-k-1}\p_n^{k+1}, \na ]f\|_{L^2(\Omega)}\le C\| f\|_{H^s(\Omega)}\le C\| u\|_{H^{s-1}(\Omega)},
\]
leading to the bound
\[
\int_{\Omega}Pu\cdot [\p_{\tau_1}^{s-k-1}\p_n^{k+1}, \na ]f dx\le C\| Pu\|_{L^2(\Omega)}\| u\|_{H^{s-1}(\Omega)}
\le C\| u\|_{H^{s}(\Omega)}\| u\|_{H^{s-1}(\Omega)}.
\]
Integration by parts as in \eqref{I3:1n} yields
\[
\la\int_{\Omega}Pu\cdot \chi_2\na \p_{\tau_1}^{s-k-1}\p_n^{k-1}\p_n^2f\ra\le C\|u\|_{H^s(\Omega)}\| u\|_{H^{s-1}(\Omega)}. 
\]
In the main term $(1-\chi_2)\na \p_{\tau_1}^{s-k-1}\p_n^{k-1}\p_n^2f$, since the support of $(1-\chi_2)$ is contained in $\Omega_{2\kappa}$, we can use \eqref{pnt2} and \eqref{eq:f:mix} to write 
\[
\begin{aligned}
\na \p_{\tau_1}^{s-k-1}\p_n^{k-1}\p_n^2f&=-\na \p_{\tau_1}^{s-k-1}\p_n^{k-1}\p_{\tau_j}^2f+\na \p_{\tau_1}^{s-k-1}\p_n^{k-1}\dive(u\cdot \na y_*)\\
&\qquad+\na \p_{\tau_1}^{s-k-1}\p_n^{k-1}\Big[\na f\cdot (n\cdot \na) n+\na f\cdot (\tau_j\cdot \na) \tau_j\Big]\\
&=-\na \p_n^{k-1}\p_{\tau_1}^{s-k-1}\p_{\tau_j}^2f-\na [\p_{\tau_1}^{s-k-1}, \p_n^{k-1}]\p_{\tau_j}^2f+\na \p_{\tau_1}^{s-k-1}\p_n^{k-1}\dive(u\cdot \na y_*)\\
&\qquad+\na \p_{\tau_1}^{s-k-1}\p_n^{k-1}\Big[\na f\cdot (n\cdot \na) n+\na f\cdot (\tau_j\cdot \na) \tau_j\Big]
\end{aligned}
\]
in $\Omega_{2\kappa}$, where the sums over $j$ were taken. Since the  commutator $[\p_{\tau_1}^{s-k-1}, \p_n^{k-1}]\p_{\tau_j}^2f$ is bounded in $H^1(\Omega)$ by $C\| f\|_{H^s(\Omega)}\le C\| u\|_{H^{s-1}(\Omega)}$ we obtain
\bq
\begin{aligned}
\left|\int_\Omega (1-\chi_2)Pu\cdot \na [\p_{\tau_1}^{s-k-1}, \p_n^{k-1}]\p_{\tau_j}^2fdx\right|  \le C\| u\|_{H^{s}(\Omega)}\| u\|_{H^{s-1}(\Omega)}.
\end{aligned}
\eq
In addition, we have
\bq
\left|\int_\Omega (1-\chi_2)Pu\cdot  \na \p_{\tau_1}^{s-k-1}\p_n^{k-1}\Big[\na f\cdot (n\cdot \na) n+\na f\cdot (\tau_j\cdot \na) \tau_j\Big]dx \right|\le C\| u\|_{H^{s}(\Omega)}\| u\|_{H^{s-1}(\Omega)}.
\eq
Thus, we are left with the two integrals
\[
\begin{aligned}
&I_1=\int_\Omega (1-\chi_2)Pu \cdot \na \p_n^{k-1}\p_{\tau_1}^{s-k-1}\p_{\tau_j}^2fdx,\\
&I_2=\int_\Omega (1-\chi_2)Pu \cdot\na \p_{\tau_1}^{s-k-1}\p_n^{k-1}\dive(u\cdot \na y_*)dx.
\end{aligned}
\]
{\it Estimate for $I_1$}. We claim that
\bq\label{est:f:mixk:0}
\|\na \p_n^{k-1}\p_{\tau_1}^{s-k-1}\p_{\tau_j}^2f\|_{L^2(\Omega)}\le \| u\|_{H^{s-1}(\Omega)}+C\| u\|_{s, k}.
\eq
First, taking $\p_{\tau_1}^{s-k-1}\p_{\tau_j}$ of \eqref{eq:f:mix} gives
\bq\label{eq:dt:f}
\begin{cases}
\begin{aligned}
\Delta \p_{\tau_1}^{s-k-1}\p_{\tau_j}f&=[\Delta, \p_{\tau_1}^{s-k-1}\p_{\tau_j}]f+[\p_{\tau_1}^{s-k-1}\p_{\tau_j},\dive](u\cdot \na y_*)+\dive \p_{\tau_1}^{s-k-1}\p_{\tau_j}(u\cdot \na y_*)\\
& :=g_1\quad\text{in}~\Omega,
\end{aligned}
\\
\p_n\p_{\tau_1}^{s-k-1}\p_{\tau_j}f=[\p_n, \p_{\tau_1}^{s-k-1}\p_{\tau_j}]f+\p_{\tau_1}^{s-k-1}\p_{\tau_j}\{(u\cdot \nabla y_*)\cdot n\}:=g_2\quad\text{on}~\p\Omega.
\end{cases}
\eq
In view of the bound
\[
\begin{aligned}
\|\na \p_n^{k-1}\p_{\tau_1}^{s-k-1}\p_{\tau_j}^2f\|_{L^2(\Omega)}&\le \|(\na \p_n^{k-1}\p_{\tau_j})\p_{\tau_1}^{s-k-1}\p_{\tau_j}f\|_{L^2(\Omega)}+\|\na \p_n^{k-1}[\p_{\tau_1}^{s-k-1}, \p_{\tau_j}]\p_{\tau_j}f\|_{L^2(\Omega)}\\
& \le C\|\p_{\tau_1}^{s-k-1}\p_{\tau_j}f\|_{H^{k+1}(\Omega)}+C\| f\|_{H^s(\Omega)}\\
& \le C\|\p_{\tau_1}^{s-k-1}\p_{\tau_j}f\|_{H^{k+1}(\Omega)}+C'\| u\|_{H^{s-1}(\Omega)}
\end{aligned}
\]
and elliptic estimates for \eqref{eq:dt:f} we have
\bq\label{est:f:mixk}
\|\na \p_n^{k-1}\p_{\tau_1}^{s-k+1}f\|_{L^2(\Omega)}\le C\|g_1\|_{H^{k-1}(\Omega)}+C\| g_2\|_{H^{k-\mez}(\p\Omega)}+C\| u\|_{H^{s-1}(\Omega)}.
\eq
 The $H^{k-1}$ norm of $g_1$ is bounded as
\[
\begin{aligned}
\| g_1\|_{H^{k-1}(\Omega)}&\le \|[\Delta, \p_{\tau_1}^{s-k-1}\p_{\tau_j}]f\|_{H^{k-1}(\Omega)}+\|[\p_{\tau_1}^{s-k-1}\p_{\tau_j},\dive](u\cdot \na y_*)\|_{H^{k-1}(\Omega)}\\
&\qquad +\|\dive \p_{\tau_1}^{s-k-1}\p_{\tau_j}(u\cdot \na y_*)\|_{H^{k-1}(\Omega)}\\
&\le C\|f\|_{H^{s}(\Omega)}+C\|u\|_{H^{s-1}(\Omega)}+C\| \p_{\tau_1}^{s-k-1}\p_{\tau_j}(u\cdot \na y_*)\|_{H^{k}(\Omega)}\\
&\le C'\|u\|_{H^{s-1}(\Omega)}+C\| \p_{\tau_1}^{s-k-1}\p_{\tau_j}(u\cdot \na y_*)\|_{H^{k}(\Omega)}.
\end{aligned}
\]
We observe that there are at most $k$ normal derivatives appearing when measure $\p_{\tau_1}^{s-k-1}\p_{\tau_j}(u\cdot \na y_*)$ in $H^{k-1}(\Omega)$, hence
\[
\| \p_{\tau_1}^{s-k-1}\p_{\tau_j}(u\cdot \na y_*)\|_{H^{k}(\Omega)}\le C\| u\|_{s, k}.
\]
Consequently
\bq\label{est:mix:g1}
\| g_1\|_{H^{k-1}(\Omega)}\le C\| u\|_{H^{s-1}(\Omega)}+C\| u\|_{s, k}.
\eq
As for $g_2$ we first use the trace theorem to have
\[
\| [\p_n, \p_{\tau_1}^{s-k-1}\p_{\tau_j}]f\|_{H^{k-\mez}(\p\Omega)}\le C\| [\p_n, \p_{\tau_1}^{s-k-1}\p_{\tau_j}]f\|_{H^{k}(\Omega)}\le C\| f\|_{H^{s}(\Omega)}\le C\| u\|_{H^{s-1}(\Omega)}.
\]
The fact that $k\ge 1$ was used in the first inequality. Then we write
\[
 \p_{\tau_1}^{s-k-1}\p_{\tau_j}(u\cdot \nabla y_*)=(\p_{\tau_1}^{s-k-1}\p_{\tau_j}u)\cdot \nabla y_*+ [\p_{\tau_1}^{s-k-1}\p_{\tau_j}, \nabla y_*\cdot ]u
 \]
 where the commutator can be bounded using the trace theorem as follows
 \[
 \|  [\p_{\tau_1}^{s-k-1}\p_{\tau_j}, \nabla y_*\cdot ] u\|_{H^{k-\mez}(\p\Omega)}\le C\|  [\p_{\tau_1}^{s-k-1}\p_{\tau_j}, \nabla y_*\cdot] u\|_{H^{k}(\Omega)}\le C\| u\|_{H^{s-1}(\Omega)}.
 \]
 In addition, 
 \[
 \begin{aligned}
\|(\p_{\tau_1}^{s-k-1}\p_{\tau_j}u)\cdot \nabla y_*\|_{H^{k-\mez}(\p\Omega)}&\le C\|\p_{\tau_1}^{s-k-1}\p_{\tau_j}u\|_{H^{k-\mez}(\p\Omega)}\\
&\le C\|\p_{\tau_1}^{s-k-1}\p_{\tau_j}u\|_{H^k(\Omega)}\\
&\le C\| u\|_{s, k}.
 \end{aligned}
 \]
Thus,
\bq\label{est:mix:g2}
\| g_2\|_{H^{k-\mez}(\p\Omega)}\le \| u\|_{H^{s-1}(\Omega)}+C\| u\|_{s, k}.
\eq
Combining \eqref{est:f:mixk}, \eqref{est:mix:g1} and \eqref{est:mix:g2} leads to the bound \eqref{est:f:mixk:0} which implies that
\bq\label{est:mix:I1}
\begin{aligned}
I_1&\le C\| Pu\|_{L^2(\Omega)}\| u\|_{H^{s-1}(\Omega)}+C\| Pu\|_{L^2(\Omega)}\| u\|_{s, k}\\
&\le  C\| u\|_{H^s(\Omega)}\| u\|_{H^{s-1}(\Omega)}+C\| u\|_{s, k+1}\| u\|_{s, k}.
\end{aligned}
\eq
{\it Estimate for $I_2$.}
Decomposing $\na ={\tau_j}\p_{\tau_j}+n\p_n $ in $\Omega_{2\kappa}\supset \supp(1-\chi_2)$ gives $I_2=I_{2a}+I_{2b}$ where
\[
\begin{aligned}
&I_{2a}=\int_\Omega (1-\chi_2)\big\{(\p_{\tau_1}^{s-k-1}\p_n^{k+1}u) \cdot {\tau_j}\big\}\big\{ \p_{\tau_j}\p_{\tau_1}^{s-k-1}\p_n^{k-1}\dive(u\cdot \na y_*)\big\}dx,\\
&I_{2b}=\int_\Omega (1-\chi_2)\big\{(\p_{\tau_1}^{s-k-1}\p_n^{k+1}u) \cdot n\big\}\big\{ \p_n \p_{\tau_1}^{s-k-1}\p_n^{k-1}\dive(u\cdot \na y_*)\big\}dx.
\end{aligned}
\]
We notice that there are at most $k$ normal derivatives in $\p_{\tau_j}\p_{\tau_1}^{s-k-1}\p_n^{k-1}\dive(u\cdot \na y_*)$, hence
\[
|I_{2a}|\le C\| u\|_{s, k+1}\| u\|_{s, k}.
\]
As for $I_{2b}$ we write using \eqref{pnt} that 
\[
\begin{aligned}
(\p_{\tau_1}^{s-k-1}\p_n^{k+1}u) \cdot n&=\p_{\tau_1}^{s-k-1}\p_n^k (\p_nu\cdot n)+[\p_{\tau_1}^{s-k-1}\p_n^k,n\cdot ] \p_nu\\
&=-\p_{\tau_1}^{s-k-1}\p_n^k (\p_{\tau_j} u\cdot \tau_j)+[\p_{\tau_1}^{s-k-1}\p_n^k,n\cdot ] \p_nu.
\end{aligned}
\]
It is readily seen that 
\[
\la\int_\Omega  (1-\chi_2) \big\{[\p_{\tau_1}^{s-k-1}\p_n^k,n\cdot ] \p_nu \big\} \big\{\p_n \p_{\tau_1}^{s-k-1}\p_n^{k-1}\dive(u\cdot \na y_*) \big\}dx\ra\le C\| u\|_{H^{s-1}(\Omega)}\| u\|_{H^s(\Omega)}.
\]
On the other hand, there are at most $k$ normal derivatives in $\p_{\tau_1}^{s-k-1}\p_n^k (\p_{\tau_j} u\cdot \tau_j)$, and thus
\[
\la \int_\Omega(1-\chi_2) \big\{\p_{\tau_1}^{s-k-1}\p_n^k (\p_{\tau_j} u\cdot \tau_j)\big\}\big\{\p_n \p_{\tau_1}^{s-k-1}\p_n^{k-1}\dive(u\cdot \na y_*) \big\}dx\ra\le C\| u\|_{s, k}\| u\|_{s, k+1}.
\]
All together we have prove that 
\bq\label{est:mix:I2}
|I_2|\le C\| u\|_{H^s(\Omega)}\| u\|_{H^{s-1}(\Omega)}+C\| u\|_{s, k+1}\| u\|_{s, k}
\eq
In view of \eqref{est:mix:I1} and \eqref{est:mix:I2} we finish the proof of \eqref{est:mix:main}, and hence the proof of \eqref{est:mix:k} with  $k+1$ in place of $k$.
\end{proof}
\subsection{$H^s$ estimate for $u$}
We have proved in Lemmas \ref{lemm:tang}, \ref{lemm:mix:1n} and \ref{lemm:mix:kn} that
\bq\label{mix:1n}
\begin{aligned}
\mez\frac{d}{dt}\| u\|^2_{s, 0}+\theta_0 \| u\|^2_{s, 0}&\le M_0\Vert u\Vert_{H^s(\Omega)}^2\Vert z\Vert_{H^s(\Omega)}+M_0\Vert u\Vert_{H^{s-1}(\Omega)}\Vert u\Vert_{H^s(\Omega)}
\end{aligned}
\eq
and
\bq\label{mix:kn}
\begin{aligned}
\mez\frac{d}{dt}\| u\|^2_{s, k}+\frac{\theta_0}{2} \| u\|^2_{s, k}&\le M_k\Vert u\Vert_{H^s(\Omega)}^2\Vert z\Vert_{H^s(\Omega)}+M_k\Vert u\Vert_{H^{s-1}(\Omega)}\Vert u\Vert_{H^s(\Omega)}\\
&\quad +M_k\| u\|_{s, k}\| u\|_{s, k-1}
\end{aligned}
\eq
for all $k\in \{1, 2,...,s\}$. Applying Young's inequality yields
\[
\begin{aligned}
M_k\| u\|_{s, j}\| u\|_{s, j-1}\le \frac{\theta_0}{4}\| u\|_{s, j-1}^2+M_j'\| u\|_{s, j-1}^2,\quad 1\le j\le s.
\end{aligned}
\]
It follows from this and \eqref{mix:kn} with $k=s$ and $k=s-1$ that
\bq\label{mix:sn}
\begin{aligned}
\frac{d}{dt}\| u\|^2_{s, s}+\frac{\theta_0}{2}\| u\|^2_{s, s}&\le 2M_s\Vert u\Vert_{H^s(\Omega)}^2\Vert z\Vert_{H^s(\Omega)}+2M_s\Vert u\Vert_{H^{s-1}(\Omega)}\Vert u\Vert_{H^s(\Omega)} +2M'_s\| u\|_{s, s-1}^2
\end{aligned}
\eq
and
\bq\label{mix:s-1n}
\begin{aligned}
\frac{d}{dt}\| u\|^2_{s, s-1}+\frac{\theta_0}{2} \| u\|^2_{s, s-1}&\le 2M_{s-1}\Vert u\Vert_{H^s(\Omega)}^2\Vert z\Vert_{H^s(\Omega)} +2M_{s-1}\Vert u\Vert_{H^{s-1}(\Omega)}\Vert u\Vert_{H^s(\Omega)}\\
&\quad+2M'_{s-1}\| u\|_{s, s-2}^2.
\end{aligned}
\eq 
Let $N_{s-1}>0$ be such that $\frac{\theta_0}{2}N_{s-1}-2M_{s-1}'=\frac{\theta_0}{2}$. Multiplying \eqref{mix:s-1n} by $N_{s-1}$ then adding the resulting  inequality to \eqref{mix:sn} we obtain
\[
\begin{aligned}
\frac{d}{dt}\big(\| u\|^2_{s, s}+N_{s-1}\| u\|^2_{s, s-1}\big)&+\frac{\theta_0}{2} \big(\| u\|^2_{s, s}+\| u\|^2_{s, s-1}\big)\le N'_{s-1}\Vert u\Vert_{H^s(\Omega)}^2\Vert z\Vert_{H^s(\Omega)}\\
&+N'_{s-1}\Vert u\Vert_{H^{s-1}(\Omega)}\Vert u\Vert_{H^s(\Omega)}+N'_{s-1}\| u\|_{s, s-2}^2
\end{aligned}
\]
for some $N'_{s-1}>0$. Continuing this process, one can find $s+1$ positive constants $B$ and $N_j$, $0\le j\le s-1$  such that
\[
\begin{aligned}
\frac{d}{dt}\big(\| u\|_{s, s}^2+\sum_{j=0}^{s-1}N_j\| u\|^2_{s, j}\big)&+\frac{\theta_0}{2}\sum_{j=0}^s\| u\|^2_{s, j}\le B\Vert u\Vert_{H^s(\Omega)}^2\Vert z\Vert_{H^s(\Omega)} +B\Vert u\Vert_{H^{s-1}(\Omega)}\Vert u\Vert_{H^s(\Omega)}.
\end{aligned}
\]
Setting
\[
Z^2(u)=\| u\|_{s, s}^2+\sum_{j=0}^{s-1}N_j\| u\|^2_{s, j}
\]
and 
\[
2\theta_1=\frac{\theta_0}{2\max_{0\le j\le s-1}\{ 1, N_j\}},
\]
we arrive at
\bq\label{mix:est:sum}
\begin{aligned}
&\frac{d}{dt}Z^2(u)+2\theta_1 Z^2(u)\le B\Vert u(t)\Vert_{H^s(\Omega)}^2\Vert z(t)\Vert_{H^s(\Omega)}+B\Vert u\Vert_{H^{s-1}(\Omega)}\Vert u\Vert_{H^s(\Omega)}.
\end{aligned}
\eq
Set 
\bq\label{def:W}
W^2(u)=Z^2(u)+\| \chi_2u\|_{H^s(\Omega)}^2
\eq
where $\chi_2$ is given by \eqref{def:chi2}. Combining \eqref{mix:est:sum} with \eqref{est:interior} one can find a constant $C>0$ such that
\bq\label{est:W}
\begin{aligned}
&\frac{d}{dt}W^2(u)+2\theta_1 W^2(u)\le C\Vert u(t)\Vert_{H^s(\Omega)}^2\Vert z(t)\Vert_{H^s(\Omega)}+C\Vert u\Vert_{H^{s-1}(\Omega)}\Vert u\Vert_{H^s(\Omega)}.
\end{aligned}
\eq
To recover the $H^s$ estimate for $u$ from the preceding estimate on $W(u)$, we prove the next lemma.
\begin{lemm}
There exists $A>0$ depending only on $s$ such that
\bq\label{HZ}
\frac{1}{A}W^2(u)\le \| u\|_{H^s(\Omega)}^2\le AW^2(u)+A\|u \|^2_{L^2(\Omega)}
\eq
for any $H^s$ vector field $u:\Omega\to \Rr^d$. 
\end{lemm}
\begin{proof}
First, the inequality 
\[
Z^2(u)\le A\| u\|_{H^s(\Omega)}^2+A\|\chi_2u\|^2_{H^s(\Omega)}
\]
is obvious if $A$ is sufficiently large.

Next recall from \eqref{decompose:na} and \eqref{nav:tn} that for any $w:\Omega\to \Rr^2$  it holds that
\bq\label{na:tn}
|\na w|^2\le |\p_n w|^2+\sum_{j=1}^d|\p_{\tau_j} w|^2+\| \chi_2 \na w\|_{L^2(\Omega)}^2.
\eq
In the rest of this proof, the sum over $j\in \{1, ...,d-1\}$ will be omitted. 
Let $D^{s-1}$ be an arbitrary partial derivative of order $s-1$. Without loss of generality, assume $D^{s-1}=\p_1D^{s-2}$ for some partial derivative $D^{s-2}$ of order $s-2$. Applying \eqref{na:tn} with $w=D^{s-1}u$ gives
\[
\begin{aligned}
\| \na D^{s-1}u\|_{L^2(\Omega)}^2
&\le \|\p_n \p_1D^{s-2}u\|_{L^2(\Omega)}^2+\|\p_{\tau_j}\p_1 D^{s-2}u\|_{L^2(\Omega)}^2+\| \chi_2 \na w\|_{L^2(\Omega)}^2.
\end{aligned}
\]
We thus have replaced one partial derivative with one normal and one tangential derivative. To continue, we commute $\p_n$ with $\p_1$ to have
\[
\begin{aligned}
\|\p_n \p_1D^{s-2}u\|_{L^2(\Omega)}^2&\le 2\|\p_1\p_n D^{s-2}u\|_{L^2(\Omega)}^2+2\|[\p_n, \p_1]D^{s-2}u\|_{L^2(\Omega)}^2\\
&\le 2\|\p_1\p_n D^{s-2}u\|_{L^2(\Omega)}^2+C\| u\|_{H^{s-1}(\Omega)}^2.
\end{aligned}
\]
Similarly for $\|\p_{\tau_j}\p_1 D^{s-2}u\|_{L^2(\Omega)}^2$ we obtain 
\[
\| \na D^{s-1}u\|_{L^2(\Omega)}^2\le 2\|\na\p_n D^{s-2}u\|_{L^2(\Omega)}^2+2\|\na\p_{\tau_j} D^{s-2}u\|_{L^2(\Omega)}^2+C\| u\|_{H^{s-1}(\Omega)}^2+\| \chi_2 u\|_{H^s(\Omega)}^2.
\]
Now applying \eqref{na:tn}  with $w=\p_n D^{s-2}u$ and $w=\p_\tau D^{s-2}u$ leads to 
\[
\begin{aligned}
\| \na D^{s-1}u\|_{L^2(\Omega)}^2&\le 2\|\p_n\p_n D^{s-2}u\|_{L^2(\Omega)}^2+2\|\p_{\tau_j}\p_n D^{s-2}u\|_{L^2(\Omega)}^2+2\|\p_n\p_{\tau_j} D^{s-2}u\|_{L^2(\Omega)}^2\\
&\quad+2\|\p_{\tau_j}\p_{\tau_j} D^{s-2}u\|_{L^2(\Omega)}^2+C\| u\|_{H^{s-1}(\Omega)}^2+\| \chi_2 u\|_{H^s(\Omega)}^2.
\end{aligned}
\]
Next we write $D^{s-2}=\p_j D^{s-3}$ with $j\in \{1,...,d\}$ and continue the process until no partial derivatives are left on the right-hand side, yielding 
\[
\| \na D^{s-1}u\|_{L^2(\Omega)}^2\le CZ^2(u)+C\| u\|_{H^{s-1}(\Omega)}^2+\| \chi_2 u\|_{H^s(\Omega)}^2.
\]
This combined with the interpolation inequality $\| u\|_{H^{s-1}}\le C\| u\|_{H^s}^\alpha\| u\|_{L^2}^{1-\alpha}$, $\alpha \in (0, 1)$ and a Young inequality implies the desired estimate \eqref{HZ}.
\end{proof}
By interpolation and Young's inequality, the last term on the right-hand side of \eqref{est:W} is bounded as
\[
C\Vert u(t)\Vert_{H^s(\Omega)}\Vert u(t)\Vert_{H^{s-1}(\Omega)}\le \gamma\Vert u(t)\Vert_{H^s(\Omega)}^2+C_\gamma \Vert u(t)\Vert_{L^2(\Omega)}^2
\]
for any $\gamma>0$. Using \eqref{HZ} and choosing $\gamma$ sufficiently small so that $A\gamma<\theta_1$, we deduce from \eqref{est:W} that
\bq\label{mix:est:sum1}
\begin{aligned}
&\frac{d}{dt}W^2(u)+\theta_1 W^2(u)\le C_2\Vert u(t)\Vert_{H^s(\Omega)}^2\Vert z(t)\Vert_{H^s(\Omega)}+C_2\Vert u(t)\Vert_{L^2(\Omega)}^2
\end{aligned}
\eq
for some $C_2>0$.
\subsection{Proof of Theorem \ref{theo:global}}
We define the bootstrap norm by
\eqref{bootstrapnorm}
\bq
\mathcal{N}(t) : =\sup_{0\le \tau\le t}\Big(\Vert z(\tau)\Vert_{H^s(\Omega)}^2+ M^2 e^{\frac{\theta_1}{2}\tau}\Vert u(\tau)\Vert_{L^2(\Omega)}^2+ M e^{\frac{\theta_1}{2}\tau}\Vert u(\tau)\Vert_{H^s(\Omega)}^2\Big)
\eq
for some large $M>0$  to be fixed and for $t<T^*$, the maximal time of existence which is positive thanks to the local existence theory, Theorem \ref{theo:lwp}. 
\begin{prop}\label{prop:btO}
There exist positive constants $\eps, C_*$, depending only on $\theta_0$ and $\Vert y_*\Vert_{H^{s+1}}$, such that whenever $\mathcal{N}(0)<\eps$, we have $\mathcal{N}(t)\le C_*\mathcal{N}(0)$ for all $t<T^*$. 
\end{prop} 
\begin{proof}
As in the proof of Proposition \ref{prop:bootstrap}, it suffices to prove that 
\begin{equation}\label{bt:ineq} \mathcal{N}(t) \le C_0\mathcal{N}(0) + C_0 \mathcal{N}(t)^{3/2} \end{equation}
for all $t<T^*$. Repeating verbatim the proof of \eqref{bd-zzz} we obtain from \eqref{est:z:O} that
\bq\label{bt:z:O}
 \| z(t)\|_{H^s}^2 
\le \| z(0)\|_{H^s}^2 + C_3M^{-1/2} (1 + \mathcal{N}(t)^{1/2}) \mathcal{N}(t).
\eq
Next we combine the estimates \eqref{u:L2:O} and \eqref{mix:est:sum1} and argue as in \eqref{bt:u} to have
\bq\label{bt:u:O}
M^2e^{\theta_1 t/2}\Vert u(t)\Vert_{L^2}^2 + M e^{\theta_1 t/2}W^2(u) 
\le \Big[ M^2\Vert u(0)\Vert_{L^2}^2 + M \| u(0)\|_{H^s}^2 \Big]  + C_5 \mathcal{N}(t)^{3/2}.
\eq
Then \eqref{bt:ineq} follows from \eqref{bt:z:O}, \eqref{bt:u:O} and \eqref{HZ}.
\end{proof}
Finally, the proof of Theorem \ref{theo:globalT}  follows along the same lines as that of Theorem \ref{theo:global}.

~\\
{\em Acknowledgement:} The authors thank Yann Brenier for introducing them the AHT model studied in this paper, Robert McCann for informing them the references \cite{Macthesis,Mac1}, and Dong Li for constructive comments. HN's research was supported by NSF grant DMS-1907776. TN's research was supported in part by the NSF under grant DMS-1764119 and by the 2018-2019 AMS Centennial Fellowship. Part of this work was done while TN was visiting the Department of Mathematics and the Program in Applied and Computational Mathematics at Princeton University.  

\bibliographystyle{abbrv}

\end{document}